\documentclass[preprint,11pt]{article}
\usepackage{appendix}
\usepackage{longtable}
\usepackage{booktabs}
\usepackage{mathrsfs}
\usepackage{amsmath,amsthm,amsfonts,color,graphicx,xcolor,dsfont}
\usepackage{pdfsync,float,hyperref} %tcolorbox, mathabx, halloweenmath
\usepackage{amssymb}
\usepackage{verbatim}
\usepackage{epstopdf}
\usepackage{color}

\usepackage{graphicx}
\usepackage{float}
\usepackage{subfigure}

\usepackage{algorithm}
\usepackage{algpseudocode}
\usepackage{amsmath}
\allowdisplaybreaks
\numberwithin{equation}{section}
\hypersetup{
    colorlinks=true,
    linkcolor=blue,
    filecolor=magenta,
    urlcolor=magenta,
  }

\newtheorem{remark}{Remark}[section]

\newtheorem{theorem}{Theorem}[section]
\newtheorem{definition}{Definition}[section]

\newtheorem{assumption}{Assumption}[section]
\newtheorem{model}{Model}

\newcommand{\cstab}[1]{{\mathrm{C}_{\mathrm{stab}}(#1)}}

\begin{document}
\title{An $hp$-Adaptive Sampling Algorithm for Dispersion Relation Reconstruction of 3D Photonic Crystals}
\date{}
\author{Yueqi Wang\thanks{Department of Mathematics, The University of Hong Kong, Pokfulam Road, Hong Kong. Email: {\tt{u3007895@connect.hku.hk}} YW acknowledges support from the Research Grants Council (RGC) of Hong Kong via the Hong Kong PhD Fellowship Scheme (HKPFS).}\and Richard Craster\thanks{Department of Mathematics, Imperial College London, London SW7 2AZ,UK. Email:\texttt{r.craster@imperial.ac.uk}} \and Guanglian Li\thanks{Corresponding author. Department of Mathematics, The University of Hong Kong, Pokfulam Road, Hong Kong. Email: {\tt{lotusli@maths.hku.hk}} GL acknowledges the support from Newton International Fellowships Alumni following-on funding awarded by The Royal Society, Young Scientists fund (Project number: 12101520) by NSFC and Early Career Scheme (Project number: 27301921), RGC, Hong Kong.}}

\maketitle

\begin{abstract}
In this work we investigate the computation of dispersion relation (i.e., band functions) for three-dimensional photonic crystals, formulated as a parameterized Maxwell eigenvalue problem, using a novel $hp$-adaptive sampling algorithm. We develop an adaptive sampling algorithm in the parameter domain such that local elements with singular points are refined at each iteration, construct a conforming element-wise polynomial space on the adaptive mesh such that the distribution of the local polynomial spaces reflects the regularity of the band functions, and define an element-wise Lagrange interpolation operator to approximate the band functions. We rigorously prove the convergence of the algorithm. To illustrate the significant potential of the algorithm, we present two numerical tests with band gap optimization.
\end{abstract}

\section{Introduction}
Photonic crystals (PhCs) are periodic nanostructures typically made from dielectric materials and are widely used in optics for wave manipulation and control \cite{joannopoulos2008molding}. For some specially designed PhCs structures, electromagnetic waves with frequencies in certain ranges cannot propagate in the crystals; these prohibited ranges are so-called band gaps and the knowledge of their position or the ability to design for specific ranges underpin many important applications, including optical transistors, photonic fibers and low-loss optical mirrors \cite{yanik2003all,russell2003photonic,wang2023analytical,labilloy1997demonstration}. The dispersion relation describes the dependence of wave frequency on its wave vector when the wave passes through certain materials. Photonic band gaps offer numerous benefits in practical applications, e.g., waveguides, cavities, and photonic crystal fibers \cite{PhysRevLett.90.123901,chen2023molecular,PhysRevB.90.115140,Li_2017}. Notably there is a parallel field in acoustics and mechanical vibration, phononic crystals \cite{laude2015book}, where band gaps and periodic media are used to filter and control elastic waves and design devices. In binary systems (i.e., PhCs composed of two materials), the appearance, size, and location of band gaps rely on various factors, including filling fraction, lattice symmetry, topology of constituent material phases, and the contrast between the permittivity of materials \cite{joannopoulos2008molding}. By effectively controlling these parameters, it is possible to design periodic composite materials that possess desired photonic band gaps, thereby achieving excellent performance in practical applications.

The behavior of electromagnetic waves is described mathematically by Maxwell's equations \cite{jackson1999classical}
and the band gaps correspond to the gaps in the spectrum of the Maxwell operator with periodic coefficients.  Thus, the task of dispersion relation reconstruction or the numerical
simulation of electromagnetic waves inside PhCs can be transformed into the approximation of the Maxwell eigenvalue problem. For
infinite crystals with perfect periodicity, Bloch's theorem reduces the task to a set of parameterized Maxwell eigenvalue problems
defined in the unit cell with periodic boundary conditions. The associated parameter is the so-called wave vector $\mathbf{k}$,
which varies in the irreducible Brillouin zone (IBZ) \cite{kuchment1993floquet}. The $n$th band function $\omega_n$, which is the
square root of the $n$th largest eigenvalue up to a constant, is regarded as a function of the wave vector $\mathbf{k}$ for all
$n\in\mathbb{N}^+$ and the band gap represents the separation between two adjacent band functions. Thus, computing the $n$th band
function $\omega_n(\mathbf{k})$ boils down to solving an infinite number of Maxwell eigenvalue problems. To create these band gaps,
it is imperative that the permittivity exhibits distinct values in both the inclusion and the background of the unit cell, and that
the magnitude difference, referred to as the ``contrast'', is often substantial. However, solving the Maxwell
eigenvalue problems with high-contrast and piecewise constant  coefficients presents notorious numerical challenges due to, e.g.,
spurious modes \cite{sun1995spurious}, and an indefinite large-scale linear system.

To reduce the computational burden, a natural
approach is to reduce  the number of parameters $\mathbf{k}$ by constraining $\mathbf{k}$ to a coarse mesh within the IBZ or its boundaries.
Nonetheless, recent research has uncovered instances where tiny band gaps, e.g., the so-called avoided crossings, observed within
this simplified approximation, vanish upon the inclusion of additional $\mathbf{k}$-points or the enhancement of eigenvalue precision
\cite{harrison2007occurrence,maurin2018probability,craster2012dangers}. This drawback greatly affects some application areas of
photonic materials in which the requirements are very sensitive to tiny band gaps. For example, it is important to avoid the
smallest band gaps in the manufacture of Large-Pitch PhC fiber for a larger output average power \cite{eidam2011fiber} while
the leakage channel fibers exploit the smallest band gaps \cite{dong2009extending}. There is a conflict between desiring high
accuracy (i.e., sampling at many $\mathbf{k}$ over the IBZ), and the computational cost in doing so for three-dimensional (3D)
structures. Thus, it is of huge importance to develop algorithms that use as few sampling points for PhCs band function
calculations as possible and yet accurately characterize the band functions.

In this work, we extend the $hp$-adaptive sampling method \cite{wang2023dispersion2} (for 2D PhCs, described by parameterized Helmholtz eigenvalue problem) to the 3D cases, by interpolating the band functions elementwise on an adaptive mesh, and demonstrate its efficiency on the band gap maximization problem \cite{chen2022inversely,yan2021photonic,hammond2022high}. First, we adaptively refine the mesh in the parameter domain, cf., Algorithm \ref{alg1}, by element-wise indicator \eqref{indicator1} and tolerance \eqref{tolerance}, which adaptively refines elements containing singularities, cf. Theorem \ref{refining mechanism}. Second, we construct a conforming finite element space \eqref{def:vn} on the adaptive mesh, and derive element-wise interpolation by identifying a set of basis functions for the local polynomial space. Third, we establish the $hp$-adaptive sampling algorithm in Algorithm \ref{alg2}, and rigorously justify its convergence in Theorems \ref{exp} and \ref{algebraic}. The resulting method allows approximating band functions accurately using significantly fewer sampling points, thus making the band gap maximization more efficient.

To demonstrate the efficiency of the algorithm, we employ two models of 3D PhCs, which consist of layered cubic structures with silicon blocks and spheres embedded in air. By optimizing the dimensions of these blocks and the radius of the spheres, we identify an optimal unit cell structure that maximizes the band gap. These experiments show the challenges of calculating the band function of 3D PhCs and showcase the process of the adaptive interpolation algorithm. To verify the performance of Algorithm \ref{alg2}, we display its convergence history and compare it with the uniform refinement.

The rest of this paper is organized as follows. In Section \ref{Problem formulation}, we summarize the derivation of band functions for 3D PhCs and discuss the properties of band functions. We also outline the challenges associated with computing the Maxwell eigenvalue problem and describe the $\mathbf{H}(\text{curl})$-conforming $\mathbf{k}$-modified N\'{e}d\'{e}lec edge finite element method (FEM). In Section \ref{sec:hp-sampling}, we extend the $hp$-adaptive sampling method \cite{wang2023dispersion2} to 3D cases, and in Section \ref{sec:convergence}, we present a convergence analysis of this method. In Section \ref{Section optimization}, we define the shape optimization problem, which aims to find the maximal band gap in cubic lattices with silicon blocks and spheres. We also include numerical experiments in this section to verify our results. Finally, in Section \ref{sec:conclusion}, we provide a conclusion and potential future works.

\section{Problem formulation}\label{Problem formulation}
We recap in this section the derivation of the band function computation for 3D PhCs as a parametrized Maxwell eigenvalue problem \eqref{strong} and its weak formulation \eqref{variational 2}. Then we describe its finite element discretization by the $\mathbf{H}(\text{curl})$- conforming $\mathbf{k}$-modified N\'{e}d\'{e}lec edge element method \eqref{discrete variational 2}. Finally, we provide an efficient approximation of the gradient of the band functions.

\subsection{Maxwell equations, eigenvalue problem, and the weak formulation}\label{Maxwell equations and eigenvalue problems, and the weak formulation}

In SI convention, the time harmonic Maxwell equations with $\mathbf{E}(\mathbf{x},t):=\mathbf{E}(\mathbf{x})\mathrm{e}^{-i\omega t}$ and $\mathbf{H}(\mathbf{x},t):=\mathbf{H}(\mathbf{x})\mathrm{e}^{-i\omega t}$ for linear, non-dispersive, and nonmagnetic media, with free charges and free currents, consist of a system of four equations \cite{jackson1999classical}
\begin{subequations}
\begin{align}
    \nabla\times\mathbf{E}(\mathbf{x})-i\omega\mu_{0}\mathbf{H}(\mathbf{x})&=0\quad\text{ in }\mathbb{R}^3,\label{harm1}    \\
    \nabla\times\mathbf{H}(\mathbf{x})+i\omega\epsilon_{0}\epsilon(\mathbf{x})\mathbf{E}(\mathbf{x})&=0\quad\text{ in }\mathbb{R}^3, \label{harm2}\\  \nabla\cdot\left(\epsilon(\mathbf{x})\mathbf{E}(\mathbf{x})\right)&=0\quad\text{ in }\mathbb{R}^3,\label{harm3}\\
    \nabla\cdot\mathbf{H}(\mathbf{x})&=0\quad\text{ in }\mathbb{R}^3,\label{harm4}
\end{align}
\end{subequations}
where $\mathbf{E}$ is the electric field, $\mathbf{H}$ is the magnetic field, the scalar $\omega\geq 0$ is the frequency of the electromagnetic wave, $\mu_0$ is the vacuum permeability, $\epsilon_0$ is the vacuum permittivity, and $\epsilon\in L^{\infty}(\mathbb{R}^3;\mathbb{R}^+)$ is the relative permittivity.

Applying the curl operator to \eqref{harm2} and using \eqref{harm1} give the equations for the magnetic field
\begin{equation}\label{H}
\left\{\begin{aligned}
        \nabla\times\left(\epsilon(\mathbf{x})^{-1}\nabla\times\mathbf{H}(\mathbf{x})\right)-\left(\omega c^{-1}\right)^2\mathbf{H}(\mathbf{x})&=0\qquad\text{ in }\mathbb{R}^3,\\
        \nabla\cdot\mathbf{H}(\mathbf{x})&=0\qquad\text{ in }\mathbb{R}^3,
\end{aligned}\right.
\end{equation}
with $\epsilon_{0}\mu_{0}=c^{-2}$, where $c$ is the speed of light. The electric field $\mathbf{E}(\mathbf{x})$ can be derived from \eqref{harm2}:
\begin{equation}
    \mathbf{E}(\mathbf{x})=\frac{i}{\omega\epsilon_0\epsilon(\mathbf{x})}\nabla\times\mathbf{H}(\mathbf{x}),
\end{equation}
which implies \eqref{harm3}. The above two equations form the governing system of 3D PhCs.

Next, 3D periodic PhCs possess a discrete translational symmetry in $\mathbb{R}^3$ \cite{joannopoulos2008molding}, i.e., the relative permittivity $\epsilon(\mathbf{x})$ satisfies
\begin{equation*}
   \epsilon(\mathbf{x}+c_1\mathbf{a}_1+c_2\mathbf{a}_2+c_3\mathbf{a}_3)=\epsilon(\mathbf{x}),
    \quad \forall \mathbf{x}\in\mathbb{R}^3\text{ and }c_1,c_2,c_3\in\mathbb{Z}.
\end{equation*}
The primitive lattice vectors, denoted as $\mathbf{a}_1,\mathbf{a}_2,\mathbf{a}_3$, are the shortest possible vectors that satisfy this condition which span the unit cell $\Omega$. For example, square lattice has primitive lattice vectors $\mathbf{a}_i = a{e}_i$ for $i=1,2,3$, with $a\in\mathbb{R}$ being the lattice constant and $(e_i)_{i=1,2,3}$ being the canonical basis in $\mathbb{R}^3$. Additionally, the reciprocal lattice vectors,  $\mathbf{b}_i$, are defined by the property
\begin{equation}
  \mathbf{b}_i \cdot \mathbf{a}_j=2\pi\delta_{ij},\quad \text{ for }i,j=1,2,3,
\end{equation}
which generate the so-called reciprocal lattice. The elementary cell of the reciprocal lattice is the (first) Brillouin zone $\mathcal{B}$, i.e., the region closer to a certain lattice point than to any other lattice points in the reciprocal lattice. Note that if the materials in the unit cell have additional symmetry, e.g. mirror symmetry, we can further restrict $\mathbf{k}$ to the irreducible Brillouin zone (IBZ).

Based on Bloch's theorem \cite{kuchment1993floquet}, the spectrum corresponding to a periodic structure in the full space is equivalent to the union of all spectra corresponding to the quasi-periodic problems in the unit cell. Thus, we apply Bloch's theorem and aim to find the wave function $\mathbf{H}$ taking the form of a plane wave modulated by a periodic function. Mathematically, it can be written as
$\mathbf{H}(\mathbf{x})=e^{i\mathbf{k}\cdot\mathbf{x}}\mathbf{u}(\mathbf{x})$,
where $\mathbf{u}(\mathbf{x})$ is a periodic function satisfying $\mathbf{u}(\mathbf{x}+\mathbf{a_i})=\mathbf{u}(\mathbf{x})$ for $i=1,2,3$ and $\mathbf{k}$ is the wave vector varying in the Brillouin zone $\mathcal{B}$. Thus, by Bloch's theorem, the eigenvalue problem \eqref{H} reduces to
\begin{equation}\label{strong}
\left\{\begin{aligned}
(\nabla+i\mathbf{k})\times\left(\epsilon(\mathbf{x})^{-1}(\nabla+i\mathbf{k})\times
\mathbf{u}(\mathbf{x})\right)-\lambda \mathbf{u}(\mathbf{x})&=\mathbf{0}\qquad\text{ in }\Omega,\\
%\label{H2}\\
(\nabla+i\mathbf{k})\cdot\mathbf{u}(\mathbf{x})&=0\qquad\text{ in }\Omega,
\end{aligned}\right.
\end{equation}
where $\lambda=\left(\omega c^{-1}\right)^2$, $\mathbf{k}$ varies in the Brillouin zone $\mathcal{B}$, and $\mathbf{u}(\mathbf{x})$ satisfies
the periodic boundary conditions $\mathbf{u}(\mathbf{x})=\mathbf{u}(\mathbf{x}+\mathbf{a_j})$ with $\mathbf{a_j}$ being the primitive lattice vector for $j=1,2,3$. For all $\mathbf{k}\in \mathcal{B}$, we can enumerate these eigenvalues in a non-decreasing manner (with their multiplicities counted) as
\begin{align*}  0\leq\lambda_1(\mathbf{k})\leq\lambda_2(\mathbf{k})\leq\cdots \leq \lambda_n(\mathbf{k})\leq\cdots\leq \infty.
\end{align*}
Then $\{\lambda_n(\mathbf{k})\}_{n=1}^{\infty}$ is an infinite sequence with $\lambda_n(\mathbf{k})$ being a continuous function with respect to the wave vector $\mathbf{k}$ and $\lambda_n(\mathbf{k})\to \infty$ when $n \to \infty$ \cite{glazman1965direct}.

Below we state the regularity of band functions, which can be proved in a similar manner as the 2D case \cite{wang2023dispersion}. Here, $ \mathrm{Lip}(\mathcal{B})$ is the space of Lipschitz continuous functions in the Brillouin zone $\mathcal{B}$ and $\mathring{A}(\mathcal{B})$ denotes the space of piecewise analytic functions with a zero Lebesgue measure of singular point set composed of points of degeneracy (i.e., the set of points at which certain band functions intersect, with zero Lebesgue measure). Throughout, the set of singular points is denoted by $\mathbf{S}$.
\begin{theorem}[Piecewise analyticity and Lipschitz continuity of band functions]\label{thm:lipchitz}
For 3D periodic PhCs, $\lambda_n(\mathbf{k})\in \mathrm{Lip}(\mathcal{B})\cap \mathring{A}(\mathcal{B})$ for all $n\in\mathbb{N}^+$.
\end{theorem}

Next, we introduce the weak formulation to \eqref{strong}. We first introduce useful notation for function spaces. For an
open, bounded and simply-connected Lipschitz domain $\Omega\subset\mathbb{R}^3$, we define the space of square integrable
functions $L^2(\Omega)$ equipped with the inner product $(u,v)=\int_{\Omega}u(\mathbf{x})\cdot\Bar{v}(\mathbf{x})\,\mathrm{d}
\mathbf{x}$ and the corresponding norm by $\|u\|=\left(\int_{\Omega}|u(\mathbf{x})|^2\,\mathrm{d}\mathbf{x}
\right)^{\frac{1}{2}}$, where $\mathbf{\Bar{\cdot}}$ denotes the complex conjugate. We denote by $H^1(\Omega)$ the usual
Sobolev space with norm
\begin{align*}
\|f\|_{H^1}:=\bigg(\sum_{|\mathbf{\alpha}|\leq1}\Big\|\frac{\partial^{\mathbf{\alpha}}}{\partial\mathbf{x}^{\mathbf{\alpha}}}f\Big\|^2_{\Omega}\bigg)^{1/2}.
\end{align*}
We also write $(\cdot,\cdot)$, $\|\cdot\|$ and $(\cdot,\cdot)_{\mathbf{H}^1}$, $\|\cdot\|_{\mathbf{H}^1}$ for the inner
products and norms of vector-valued functions spaces $L^2(\Omega)^3$, $H^1(\Omega)^3$, respectively. Let $(\mathbf{u},
\mathbf{v})_{\mathbf{H}(\text{curl};\Omega)}=(\mathbf{u},\mathbf{v})+(\nabla\times\mathbf{u},\nabla\times\mathbf{v})$
be the standard inner product on the Sobolev space $\mathbf{H}(\text{curl};\Omega):=\{\mathbf{v}\in L^2(\Omega)^3:
\nabla\times\mathbf{v}\in L^2(\Omega)^3\}$ and $(\mathbf{u},\mathbf{v})_{\mathbf{H}(\text{div};\Omega)}=(\mathbf{u},\overline{\mathbf{v}})+(\nabla\cdot\mathbf{u},\nabla\cdot\mathbf{v})$ be the inner product on $\mathbf{H}(\text{div};\Omega):=\{\mathbf{v}\in L^2(\Omega)^3:\nabla\cdot\mathbf{v}\in L^2(\Omega)^3\}$. Also we will use the following spaces for the weak formulation:
\begin{equation*}
\begin{aligned}
    C^\infty({\Omega})&=\{v:{\Omega}\to\mathbb{C}:D^{\mathbf{\alpha}}v \text{ exists },\forall\text{multi-indices }\mathbf{\alpha}\},\\
    C^\infty_{\text{per}}({\Omega})&=\{v\in C^\infty({\Omega}):v(\mathbf{x}+\mathbf{a}_i)=v(\mathbf{x}),\forall i=1,2,3,\forall \mathbf{x},\mathbf{x}+\mathbf{a}_i\in\partial\Omega\},\\
    C^\infty_{\text{div}}(\Bar{\Omega})&=\{\mathbf{v}\in C^\infty(\Bar{\Omega})^3:\mathbf{v}(\mathbf{x}+\mathbf{a}_i)\cdot\mathbf{n}=-\mathbf{v}(\mathbf{x})\cdot\mathbf{n},\forall i=1,2,3,\forall \mathbf{x},\mathbf{x}+\mathbf{a}_i\in\partial\Omega\},\\
    C^\infty_{\text{curl}}({\Omega})&=\{\mathbf{v}\in C^\infty({\Omega})^3:\mathbf{v}(\mathbf{x}+\mathbf{a}_i)\times\mathbf{n}=-\mathbf{v}(\mathbf{x})\times\mathbf{n},\forall i=1,2,3,\forall \mathbf{x},\mathbf{x}+\mathbf{a}_i\in\partial\Omega\}.
\end{aligned}
\end{equation*}
Here $\mathbf{n}$ is the unit outward normal vector to $\partial\Omega$. The periodic Sobolev spaces $H_{\text{per}}^1(\Omega)$, $\mathbf{H}_{\text{per}}(\text{div};\Omega)$, $\mathbf{H}_{\text{per}}(\text{curl};\Omega)$ are the closures of the spaces $C^\infty_{\text{per}}({\Omega})$, $C^\infty_{\text{div}}({\Omega})$, $C^\infty_{\text{curl}}({\Omega})$ with respect to the Sobolev space norms $\|\cdot\|_{H^1}$,$\|\cdot\|_{\mathbf{H}(\text{div};\Omega)}$, $\|\cdot\|_{\mathbf{H}(\text{curl};\Omega)}$ induced by the inner product mentioned above respectively. The periodic space $L^2_{\text{per}}(\Omega)$ is the closure of the space $C^\infty_{\text{per}}({\Omega})$ with respect to the norm $\|\cdot\|$.

Then the spaces for the weak formulation are given by
\begin{equation}\label{space 1}
Q:=H_{\text{per}}^1(\Omega)\quad\mbox{and}\quad
    \mathbf{V}:=\mathbf{H}_{\text{per}}(\text{curl};\Omega).
\end{equation}
For $\mathbf{u},\mathbf{v}\in \mathbf{V}$ and $q\in Q$, we define the sesquilinear forms
\begin{align*}
a(\mathbf{u},\mathbf{v};\mathbf{k})&:=\int_{\Omega}\epsilon(\mathbf{x})^{-1}(\nabla+i\mathbf{k})\times \mathbf{u}\cdot\overline{(\nabla+i\mathbf{k})\times \mathbf{v}}\,\mathrm{d}\mathbf{x},\\
b(q,\mathbf{u};\mathbf{k})&:=-\int_{\Omega}(\nabla+i\mathbf{k}) q\cdot \bar{\mathbf{u}}\,\mathrm{d}\mathbf{x},\\
    (\mathbf{u},\mathbf{v})&:=\int_{\Omega} \mathbf{u}\cdot\bar{\mathbf{v}}\,\mathrm{d}\mathbf{x}.
\end{align*}
\begin{comment}
Defining $\mathbf{V}^0:=\{\mathbf{u}\in V:\overline{b(q,\mathbf{u})}=0\quad\forall q\in Q\}$, the variational formulation of \eqref{strong} can now be stated as: for a given $\mathbf{k}\in\mathcal{B}$, find non-trivial eigenpair $(\lambda,\mathbf{u}) \in (\mathbb{R},\mathbf{V}^0)$ satisfying
\begin{equation}\label{variational 1}
\int_{\Omega}\epsilon(\mathbf{x})^{-1}(\nabla+i\mathbf{k})\times \mathbf{u}\cdot\overline{(\nabla+i\mathbf{k})\times \mathbf{v}}-\lambda \mathbf{u}\cdot\Bar{\mathbf{v}}\,\mathrm{d}\mathbf{x}=0, \text{ for all }v \in \mathbf{V}^0.
\end{equation}
Using the sesquilinear forms, it reads: for a given $\mathbf{k}$ in $\mathcal{B}$, find non-trivial eigenpair $(\lambda,\mathbf{u}) \in (\mathbb{R},\mathbf{V}^0)
$ such that
\begin{equation}\label{simply}
a(\mathbf{u},\mathbf{v})=\lambda (\mathbf{u},\mathbf{v}), \text{ for all }v \in \mathbf{V}^0.
\end{equation}

\begin{remark}
The bilinear form $a(\cdot,\cdot)$ is positive semidefinite on $\mathbf{V}\times\mathbf{V}$ and is positive definite on $\mathbf{V}^0\times\mathbf{V}^0$ for $\mathbf{k}\neq 0$. Besides, $(\cdot,\cdot)$ is positive definite on $\mathbf{V}\times\mathbf{V}$. Note that we always assume that $\mathbf{k}\neq 0$, since when $\mathbf{k}=0$, it is a special case which leads to some complications. Also in the computations, we can avoid that by approximating $\mathbf{k}=0$ with a vector of a very small magnitude.
\end{remark}
\end{comment}
Then the variational formulation to \eqref{strong} can be written in a mixed form: for a given $\mathbf{k}\in\mathcal{B}$, find non-trivial eigentriplets $\{(\lambda_n,\mathbf{u}_n,p_n)\}_{n=1}^{\infty}\subset(\mathbb{R},\mathbf{V},Q)$ such that
\begin{equation}\label{variational 2}
\left\{
\begin{aligned}
a(\mathbf{u}_n,\mathbf{v};\mathbf{k})+b(p_n,\mathbf{v};\mathbf{k})&=\lambda_n (\mathbf{u}_n,\mathbf{v}), &\forall v \in \mathbf{V},\\
\overline{b(q,\mathbf{u}_n;\mathbf{k})}&=0, &\forall q \in Q.
    \end{aligned}
\right.
\end{equation}
Finally, we introduce a regularity assumption for the finite element analysis \cite[Assumption $\operatorname{H}_0$]{boffi2006modified},
\begin{assumption}\label{ass:reg}
There exists $s>1/2$ such that $\mathbf{V}\cap \mathbf{H}_{\mathrm{per}}(\mathrm{div};\Omega)$ is continuously embedded in $H^{s}(\Omega)^3$. Moreover, for a given $\mathbf{k}\in\mathcal{B}$ and for any $\mathbf{f}\in L^2(\Omega)^3$, let $(\mathbf{u},p)\in(\mathbf{V},Q)$ satisfy
\begin{equation*}%\label{source problem 1}
\left\{
\begin{aligned}
a(\mathbf{u},\mathbf{v};\mathbf{k})+b(p,\mathbf{v};\mathbf{k})&= (\mathbf{f},\mathbf{v}), &\forall v \in \mathbf{V},\\
\overline{b(q,\mathbf{u};\mathbf{k})}&=0, &\forall q \in Q.
    \end{aligned}
\right.
\end{equation*}
Then there is $s>1/2$ and $r\in (0,1/2)$, such that $\mathbf{u}\in H^{s}(\Omega)^3$ and $\nabla\times\mathbf{u}\in H^{r}(\Omega)^3$.
\end{assumption}

\subsection{\texorpdfstring{$\mathbf{H}(\text{curl})$}{Lg}- conforming \texorpdfstring{$\mathbf{k}$}{Lg}-modified N\'{e}d\'{e}lec edge FEM}\label{edge FEM}
In this section we describe the discretization of problem \eqref{variational 2} for each wave vector $\mathbf{k}\in\mathcal{B}$ using the $\mathbf{H}(\text{curl})$- conforming $\mathbf{k}$-modified N\'{e}d\'{e}lec edge element method, which is stable with respect to $\mathbf{k}$.
Since the divergence of the curl of any vector is zero,  taking the divergence on both sides of the first equation in \eqref{strong} gives its second equation. Furthermore, the zero eigenvalue of \eqref{strong} has an infinite multiplicity
since the curl operator has an infinite-dimensional kernel. Thus, a discretization of the eigenvalue problem \eqref{strong} solely using the scalar node-based FEM may contain non-physical solutions (i.e., spurious modes \cite{sun1995spurious}). Indeed, most eigenvalue solvers are inefficient when dealing with such eigenvalue problems \cite{bai2000templates}.

To eliminate these spurious modes, the standard approach is to use the edge FEM and add the constraint \eqref{strong} explicitly \cite{lee1991tangential,sun1995spurious}. The Lagrange multiplier in the mixed form in \eqref{variational 2} is one common approach to add the divergence constraint to the system. Also, several methods have been proposed to eliminate this large null space \cite{dobson2000efficient,chou2019finite,huang2019isira,
lyu2021fame,lu2022parallel,lu2017discontinuous}, including the
special $\mathbf{H}(\text{curl})$-conforming $\mathbf{k}$-modified N\'{e}d\'{e}lec edge elements \cite{dobson2001analysis,boffi2006modified} and the mixed discontinuous Galerkin (DG) formulation with modified N\'{e}d\'{e}lec basis functions \cite{lu2017discontinuous}. The $\mathbf{H}(\text{curl})$- conforming $\mathbf{k}$-modified N\'{e}d\'{e}lec edge element method was first proposed and analyzed in \cite{dobson2000efficient,dobson2001analysis} under strong regularity restrictions and in the case of uniform mesh sequences, with an improved analysis under minimal assumptions on the regularity of the eigensolutions and on the mesh sequences \cite{boffi2006modified}. We recap the N\'{e}d\'{e}lec edge element method below.

Let $\mathcal{T}_h$ be a periodic, shape-regular, conformal tetrahedral mesh for $\Omega$ aligned with the possible discontinuities of $\alpha,\beta$, where the phrase ``periodic" means that the meshes are the same on each pair of the boundary.  Let $h=\max_{K\in\mathcal{T}_h}h_K$, where $h_K$ is the diameter of $K$. We denote the set of all edges by $\mathcal{E}$ and the set of all vertexes by $\mathcal{V}$. The conforming finite element spaces of \eqref{space 1} are \cite{boffi2006modified}
\begin{equation}\label{space 2}
    \begin{aligned}
        Q^\mathbf{k}_h&:=\{q\in Q:q|_K=e^{-i\mathbf{k}\cdot\mathbf{x}}\Tilde{q},\;\quad \forall \Tilde{q}\in P_{k+1}(K) \text{ and }\forall K\in\mathcal{T}_h\},\\
    \mathbf{V}^\mathbf{k}_h&:=\{\mathbf{v}\in \mathbf{V}:\mathbf{v}|_K=e^{-i\mathbf{k}\cdot\mathbf{x}}\Tilde{\mathbf{v}},\quad \forall \Tilde{\mathbf{v}}\in S_{k}(K)\;\;\text{ and }\forall K\in\mathcal{T}_h\},
    \end{aligned}
\end{equation}
where $P_k(K)$ is the space of local polynomials of total degree less than or equal to $k$ on $K$, and $S_k(K)$ is the space of vector polynomials with the form $\mathbf{a}(\mathbf{x})+\mathbf{b}(\mathbf{x})\times\mathbf{x}$ with $\mathbf{a},\mathbf{b}\in P_k(K)^3$.

Given $\mathbf{k}\in\mathcal{B}$, the discretization of problem \eqref{variational 2} reads: find a non-trivial eigentriplet $(\lambda_{n,h},\mathbf{u}_{n,h},p_{n,h})\in(\mathbb{R},\mathbf{V}_h^\mathbf{k},Q_h^\mathbf{k})$ with $(\mathbf{u}_{n,h},\lambda_{n,h})\neq(\mathbf{0},0)$, such that
\begin{equation}\label{discrete variational 2}
\left\{\begin{aligned}
a(\mathbf{u}_{n,h},\mathbf{v}; \mathbf{k})+b(p_{n,h},\mathbf{v};\mathbf{k})&=\lambda_{n,h} (\mathbf{u}_{n,h},\mathbf{v}),\quad\forall \mathbf{v} \in \mathbf{V}_h^\mathbf{k},\\
\overline{b(q,\mathbf{u}_{n,h};\mathbf{k})}&=0,\qquad\qquad\qquad\forall q \in Q_h^\mathbf{k}.
\end{aligned}\right.
\end{equation}
For $\mathbf{k}=\mathbf{0}$, one should make the spaces $\mathbf{V}$, $Q$ and their corresponding FEM spaces $\mathbf{V}^\mathbf{0}_h$, $Q^\mathbf{0}_h$ orthogonal to constant vectors or constants in order to guarantee the stability of \eqref{variational 2} and \eqref{discrete variational 2} \cite{boffi1999computational}.

\begin{theorem}[{\cite[Theorem 3]{boffi2006modified}}]
 Let $\lambda_n$ be an eigenvalue of problem \eqref{variational 2} with multiplicity $m$. Then there are $m$ discrete eigenvalues $\lambda_{n_1,h},\cdots,\lambda_{n_m,h}$ of problem \eqref{discrete variational 2} converging to $\lambda$. Moreover, with $\hat{\lambda}_n:=\frac{1}{m}\sum_{i=1}^m\lambda_{n_i,h}$, there is $h_0$ such that for any $0<h<h_0$ there holds \begin{equation}
       |\lambda_n-\hat{\lambda}_n|\leq C h^{2r},
   \end{equation}
where $C>0$ is a constant and $r$ is defined in Assumption \ref{ass:reg}.
\end{theorem}

\subsection{Gradient of \texorpdfstring{$\lambda_n(\mathbf{k})$}{Lg} in the parameter domain $\mathcal{B}$}\label{section gradient}
Next, we derive the formula of the gradient $\partial_j\lambda_n:=\frac{\partial\lambda_n}{\partial k_j}$ for $n\in\mathbb{N}_+$ and $j=1,2,3$, and provide its numerical approximation. Recall that $(e_j)_{j=1,2,3}$ denotes the canonical basis in $\mathbb{R}^3$.

\begin{theorem}
Let the non-trivial eigentriplet $(\lambda_n(\mathbf{k}),\mathbf{u}_n(\mathbf{k}),p(\mathbf{k}))\in(\mathbb{R},\mathbf{V},Q)$ be the solution to the weak formulation \eqref{variational 2} for any $\mathbf{k}\in\mathcal{B}$. Then there holds
\begin{align}\label{eq:partial-lambda}
\partial_j\lambda_n=-2\Im\Big(\int_{\Omega}\epsilon^{-1}(e_j\times\mathbf{u}_n)\cdot\overline{(\nabla+i
\mathbf{k})\times\mathbf{u}_n}\,\mathrm{d}\mathbf{x}\Big),\quad j=1,2,3.
\end{align}
\end{theorem}
\begin{proof}
We abbreviate $\lambda_n(\mathbf{k})$ and $\mathbf{u}_n(\mathbf{k})$ to $\lambda$ and $\mathbf{u}$.
By definition,  for all $\mathbf{v}\in \mathbf{V}$, there holds
\begin{equation*}
\begin{aligned}
    a(\mathbf{u},\mathbf{v};\mathbf{k})&=\int_{\Omega}\epsilon^{-1}(\nabla+i\mathbf{k})\times \mathbf{u}\cdot\overline{(\nabla+i\mathbf{k})\times \mathbf{v}}\,\mathrm{d}\mathbf{x}\\
    &=\int_{\Omega}\epsilon^{-1}(\nabla+i\mathbf{k})\times \mathbf{u}\cdot(\nabla-i\mathbf{k})\times \overline{\mathbf{v}}\,\mathrm{d}\mathbf{x}\\
&=\int_{\Omega}\epsilon^{-1}(\nabla\times\mathbf{u}+i\mathbf{k}\times \mathbf{u})\cdot(\nabla\times\overline{\mathbf{v}}-i\mathbf{k}\times \overline{\mathbf{v}})\,\mathrm{d}\mathbf{x}\\
&=\int_{\Omega}\epsilon^{-1}(\nabla\times\mathbf{u})\cdot(\nabla\times\overline{\mathbf{v}})\,\mathrm{d}\mathbf{x}-\int_{\Omega}\epsilon^{-1}(\nabla\times\mathbf{u})\cdot(i\mathbf{k}\times\overline{\mathbf{v}})\,\mathrm{d}\mathbf{x}\\
&\quad+\int_{\Omega}\epsilon^{-1}(i\mathbf{k}\times\mathbf{u})\cdot(\nabla\times\overline{\mathbf{v}})\,\mathrm{d}\mathbf{x}+\int_{\Omega}\epsilon^{-1}(\mathbf{k}\times\mathbf{u})\cdot(\mathbf{k}\times\overline{\mathbf{v}})\,\mathrm{d}\mathbf{x}.
%&:=a_1(\mathbf{u},\mathbf{v})+a_2(\mathbf{u},\mathbf{v};\mathbf{k})+a_3(\mathbf{u},\mathbf{v};\mathbf{k})+a_4(\mathbf{u},\mathbf{v};\mathbf{k}).
\end{aligned}
\end{equation*}
Upon utilizing the formula for $j=1,2,3$,
\begin{equation*}
\partial_j(\mathbf{k}\times \mathbf{w})=e_j\times\mathbf{w}+\mathbf{k}\times \partial_j\mathbf{w},\quad \forall \mathbf{w}\in \mathbf{V},
    %(\nabla+i\mathbf{k})\times \mathbf{u}=((\partial_yu_3-\partial_zu_2)+(k_2u3-k_3u_2)i,(\partial_zu_1-\partial_xu_3)+(k_3u1-k_1u_3)i,(\partial_xu_2-\partial_yu_1)+(k_1u2-k_2u_1)i),
\end{equation*}
and taking the partial derivative of \eqref{variational 2} with respect to $k_j$, we obtain
\begin{equation*}
    a(\partial_j\mathbf{u},\mathbf{v};\mathbf{k})-\lambda (\partial_j\mathbf{u},\mathbf{v})=f(\mathbf{v}),
\end{equation*}
with $f(\mathbf{v};\mathbf{k}):=-a_2^j(\mathbf{u},\mathbf{v})-a_3^j(\mathbf{u},\mathbf{v})-a_4^j(\mathbf{u},\mathbf{v};\mathbf{k})+\partial_j\lambda (\mathbf{u},\mathbf{v})$, where $a_2^j(\cdot,\cdot), a_3^j(\cdot,\cdot)$ and $a_4^j(\cdot,\cdot)$ are defined as
\begin{equation*}
	\begin{aligned}
		 a_2^j(\mathbf{u},\mathbf{v})&:=-\int_{\Omega}\epsilon^{-1}(\nabla\times\mathbf{u})\cdot(ie_j\times\overline{\mathbf{v}})\,\mathrm{d}\mathbf{x},\\
		 a_3^j(\mathbf{u},\mathbf{v})&:=\int_{\Omega}\epsilon^{-1}(ie_j\times\mathbf{u})\cdot(\nabla\times\overline{\mathbf{v}})\,\mathrm{d}\mathbf{x},\\
       a_4^j(\mathbf{u},\mathbf{v};\mathbf{k})&:=\int_{\Omega}\epsilon^{-1}(e_j\times\mathbf{u})\cdot(\mathbf{k}\times\overline{\mathbf{v}})\,\mathrm{d}\mathbf{x}+\int_{\Omega}\epsilon^{-1}(\mathbf{k}\times\mathbf{u})\cdot(e_j\times\overline{\mathbf{v}})\,\mathrm{d}\mathbf{x}.
       \end{aligned}
\end{equation*}
By repeating the argument for \cite[Theorem 3.6]{wang2023dispersion}, we can derive the derivatives of the eigenvalue $\lambda$ with respect to $\mathbf{k}$ through $f(\mathbf{u})=0$, and arrive at
\begin{equation} \label{partial 1}
\partial_j\lambda=a_2^j(\mathbf{u},\mathbf{u})+a_3^j(\mathbf{u},\mathbf{u})+a_4^j(\mathbf{u},\mathbf{u};\mathbf{k}),
\end{equation}
since the eigenfunction is normalized, i.e., $(\mathbf{u},\mathbf{u})=1$.
Following equation \eqref{partial 1}, we find
\begin{equation*}
	\begin{aligned}
\partial_j\lambda&=\int_{\Omega}\epsilon^{-1}\left(-(\nabla\times\mathbf{u})\cdot(ie_j\times\overline{\mathbf{u}})+(ie_j\times\mathbf{u})\cdot(\nabla\times\overline{\mathbf{u}})+(e_j\times\mathbf{u})\cdot(\mathbf{k}\times\overline{\mathbf{u}})+(\mathbf{k}\times\mathbf{u})\cdot(e_j\times\overline{\mathbf{u}})\right)\,\mathrm{d}\mathbf{x}\\
		 &=\int_{\Omega}\epsilon^{-1}\left((e_j\times\mathbf{u})\cdot(i\nabla+\mathbf{k})\times\overline{\mathbf{u}}+(e_j\times\overline{\mathbf{u}})\cdot(-i\nabla+\mathbf{k})\times\mathbf{u}\right)\,\mathrm{d}\mathbf{x}\\
		&=\int_{\Omega}\epsilon^{-1} \left((e_i\times\mathbf{u})\cdot(i\nabla+\mathbf{k})\times\overline{\mathbf{u}}+\overline{(e_i\times\mathbf{u})\cdot(i\nabla+\mathbf{k})\times\overline{\mathbf{u}}}\right)\,\mathrm{d}\mathbf{x}\\
		 &=2\Re\left(\int_{\Omega}\epsilon^{-1}(e_j\times\mathbf{u})\cdot(i\nabla+\mathbf{k})\times\overline{\mathbf{u}}\,\mathrm{d}\mathbf{x}\right)\\
		 &=-2\Im\left(\int_{\Omega}\epsilon^{-1}(e_j\times\mathbf{u})\cdot\overline{(\nabla+i\mathbf{k})\times\mathbf{u}}\,\mathrm{d}\mathbf{x}\right).
	\end{aligned}
\end{equation*}
This completes the proof.
\end{proof}

The use of modified basis functions \eqref{space 2} greatly simplifies the computation of $\partial_j\lambda_n$, and hence reduce the computational complexity. For example, for the lowest order N\'{e}d\'{e}lec FEM space with $k=0$ in \eqref{space 2}, the degrees of freedom for $\mathbf{V}_h^{\mathbf{k}}$ are
	\begin{equation*}
		\ell_{e,\mathbf{k}}(\mathbf{v}):=\int_{\mathbf{x}_e}^{\mathbf{y}_e}e^{i\mathbf{k}\cdot(x-\mathbf{z}
			_e)}\mathbf{v}\cdot\mathbf{t}_e \,\mathrm{d}s,\quad \text{for all }e\in\mathcal{E}\text{ and }\mathbf{v}\in \mathbf{V}_h^{\mathbf{k}}.
	\end{equation*}
	Here, $\mathbf{z}_e:=(\mathbf{x}_e+\mathbf{y}_e)/2$ denotes the midpoint of edge $e\in\mathcal{E}$, identified with the ordered vertices $(\mathbf{x}_e,\mathbf{y}_e)$, and $\mathbf{t}_e:=(\mathbf{y}_e-\mathbf{x}_e)/\|\mathbf{y}_e-\mathbf{x}_e\|$ represents the unit tangent vector of $e$. Hence, $\mathbf{V}_h^{\mathbf{k}}$ of the lowest order can be expressed as
	\begin{equation*}
		\mathbf{V}_h^{\mathbf{k}}:=\text{span}\{\psi_{e,\mathbf{k}}=e^{-i\mathbf{k}\cdot(x-\mathbf{z}_e)}\psi_{e,\mathbf{0}}:e\in\mathcal{E}\},
	\end{equation*}
	with $\ell_{e,\mathbf{k}}(\psi_{e,\mathbf{k}})=1$ and $\ell_{e,\mathbf{0}}(\psi_{e,\mathbf{0}})=1$. Similarly, $Q_h^{\mathbf{k}}$ of the lowest-order can be represented by
 \begin{align*}
 Q_h^{\mathbf{k}}:=\text{span}\{q_{v,\mathbf{k}}=e^{-i\mathbf{k}\cdot(x-\mathbf{x}_v)}q_{v,\mathbf{0}}:v\in\mathcal{V}\},
 \end{align*}
with $\ell_{v,\mathbf{k}}(q_{v,\mathbf{k}})=1$ and $\ell_{v,\mathbf{0}}(q_{v,\mathbf{0}})=1$, which has the vertex-based degrees of freedom defined as $\ell_{v,\mathbf{k}}(q):=e^{i\mathbf{k}\cdot(x-\mathbf{x}_v)}q(\mathbf{x}_v)$, for all $v\in\mathcal{V}$, $\mathbf{x}_v$ is the coordinate of a vertex in $\mathcal{V}$.
	
With the basis functions defined above, we obtain
\begin{align*}  a(\psi_{e_1,\mathbf{k}},\psi_{e_2,\mathbf{k}};\mathbf{k})&=e^{i\mathbf{k}\cdot(\mathbf{z}_{e_1}-\mathbf{z}_{e_2})}\int_{\Omega}\epsilon^{-1}\nabla\times \psi_{e_1,\mathbf{0}}\cdot\overline{\nabla\times \psi_{e_2,\mathbf{0}}}\,\mathrm{d}\mathbf{x},\\
			b(q_{v,\mathbf{k}},\psi_{e,\mathbf{k}};\mathbf{k})&:=-e^{i\mathbf{k}\cdot(\mathbf{x}_{v}-\mathbf{z}_{e})}\int_{\Omega}\nabla q_{v,\mathbf{0}}\cdot \overline{\psi_{e,\mathbf{0}}}\,\mathrm{d}\mathbf{x},\\
			(\psi_{e_1,\mathbf{k}},\psi_{e_2,\mathbf{k}})&:=e^{i\mathbf{k}\cdot(\mathbf{z}_{e_1}-\mathbf{z}_{e_2})}\int_{\Omega} \psi_{e_1,\mathbf{0}}\cdot\overline{\psi_{e_2,\mathbf{0}}}\,\mathrm{d}\mathbf{x}.
\end{align*}
Together with \eqref{eq:partial-lambda}, we obtain
\begin{align*} \partial_j\lambda_{n,h}=\int_{\Omega}\epsilon^{-1}(e_i\times\psi_{e_1,\mathbf{k}})\cdot\overline{(\nabla+i\mathbf{k}) \times\psi_{e_2,\mathbf{k}}}\,\mathrm{d}\mathbf{x}& =e^{i\mathbf{k}\cdot(\mathbf{z}_{e_1}-\mathbf{z}_{e_2})}\int_{\Omega}\epsilon^{-1}(e_i\times\psi_{e_1,\mathbf{0}}) \cdot\overline{\nabla\times\psi_{e_2,\mathbf{0}}}\,\mathrm{d}\mathbf{x}.
\end{align*}
Thus, the use of the modified basis functions separates the integral with the parameter $\mathbf{k}$, transforming it into a standalone exponential function during the computation. This separation streamlines evaluating the integrals, making it a conventional integration involving parameter-free polynomials.

\section{\texorpdfstring{$hp$}{Lg}-adaptive sampling method in the parameter domain $\mathcal{B}$}\label{sec:hp-sampling}
In this section, we generalize the $hp$-adaptive sampling method \cite{wang2023dispersion2} from 2D PhCs to 3D PhCs for the purpose of computing the $\ell$th and $\ell+1$th band functions with $1\leq \ell\in \mathbb{N}$. We introduce in Section \ref{Mesh design} an adaptive mesh in the parameter domain $\mathcal{B}$ which refines the elements containing singularity points, and propose in Section \ref{element-wise interpolation} an element-wise interpolation operator to reconstruct the target band functions.
\subsection{Adaptive mesh generation}\label{Mesh design}
We first recall the definition of a regular mesh without hanging nodes \cite{schwab1998p}. Let $\mathcal{B}\subset \mathbb{R}^3$ be a polygon with straight sides and a mesh $\mathcal{T}=\{T_i\}_{i=1}^M$ be a partition of $\mathcal{B}$ into $M$ open, disjoint tetrahedrons $T_i$ with $\overline{D}=\bigcup_{i=1}^M\overline{T_i}$. For each $T\in \mathcal{T}$, we denote by $h_T$ its diameter and $\rho_T$ the diameter of its largest inscribed ball.
\begin{definition}\label{regular mesh}
 A mesh $\mathcal{T}=\{T_i\}_{i=1}^M$ is said to be a regular mesh without hanging nodes if it satisfies the following conditions.
(i) For $i\neq j$, $\overline{T_i}\bigcap\overline{T_j}$ is either empty or it consists of a vertex or an entire edge or an entire face of $T_i$ and (ii) there exists a constant $C>0$ such that $h_T\leq C\rho_T$.
\end{definition}
Now we develop a strategy to generate a family of nested regular tetrahedral meshes $\mathcal{T}_n$, $n=1,2,\cdots$, such
that $\mathcal{T}_{n+1}$ is a refinement of $\mathcal{T}_n$. To this end, we introduce several notations for the given regular tetrahedral mesh $\mathcal{T}_n$. The mesh size of $\mathcal{T}_n$ is defined as $h_n:=\max_{T\in\mathcal{T}_n}h_T$. The collection of all the faces, edges, and vertices over $\mathcal{T}_n$ is denoted as $\mathcal{F}_n$, $\mathcal{E}_n$ and $\mathcal{V}_n$, respectively. For any $T\in\mathcal{T}_n$, we define its faces as $\mathcal{F}_T:=\partial T\bigcap \mathcal{F}_n$, and the $j$th face is denoted by $F_T^j$ for $j=1,2,3,4$. Similarly, we define its edges and vertices as $\mathcal{E}_T:=\partial T\bigcap \mathcal{E}_n$ and $\mathcal{V}_T:=\partial T\bigcap \mathcal{V}_n$, and the $j$th edge and vertex are denoted by $E_T^j$ for $j=1,\cdots,6$, and $V_T^j$ for $j=1,2,3,4$. For any $F\in\mathcal{F}_n$, we define the collection of elements having $F$ as a face as $\mathcal{T}_F:=\bigcup\{T\in\mathcal{T}_n:F\in\mathcal{F}_T\}$ and the collection of elements having $E$ as an edge as $\mathcal{T}_E:=\bigcup\{T\in\mathcal{T}_n:E\in\mathcal{E}_T\}$. Finally, we define $P_m(T)$ as the space of polynomial functions with total degrees at most $m\in\mathbb{N}^+$ on element $T$, and $P_m(F)$, $P_m(E)$ as the space of polynomial functions of degrees at most $m\in\mathbb{N}^+$ on each face $F\in\mathcal{F}_n$ and each edge $E\in\mathcal{E}_n$, respectively.

Theorem \ref{thm:lipchitz} implies that each $\lambda_n(\mathbf{k})$ is a piecewise analytic function of the wave vector $\mathbf{k}$, with singular points at the branch points. Our mesh design strategy aims to identify elements within $\mathcal{T}_n$ where $\lambda_{\ell}$ and $\lambda_{\ell+1}$ are not smooth. To accomplish this, we introduce an element-wise indicator
\begin{equation}\label{indicator1}
\eta(T):=\min_{\substack{\max\{1,\ell-1\}\leq q \leq \ell+1\\\mathbf{k}\in\mathcal{V}_T}}\left\{|\lambda_q-\lambda_{q+1}| \right\}\quad \text{for }T\in \mathcal{T}_n,
\end{equation}
 which uses the adjacent values on the vertices of each tetrahedron element to ascertain the regularity of the eigenvalues as functions of $\mathbf{k}$ within each element. Furthermore, the element-wise tolerance is
\begin{equation}\label{tolerance}
     \operatorname{tol}_1(T)=\kappa h_T C_{\max,\ell}(T),
 \end{equation}
where $\kappa\geq 2\sqrt{3} $ is a constant and
\begin{align*}
C_{\max,\ell}(T):=\sup_{\substack{\max\{1,\ell-1\}\leq q \leq \ell+1\\ \mathbf{k}\in T\backslash\mathbf{S}}}\left|\nabla\lambda_q(\mathbf{k})\right|
\end{align*}
is an upper bound of the gradients of $\lambda_{q}(\mathbf{k})$ with $\max\{1,\ell-1\}\leq q \leq \ell+1$.

\begin{remark}
In the numerical experiments, we take $\kappa=2\sqrt{2}$, which represents a more relaxed tolerance. It reduces the computational cost, while ensuring the expected convergence rate and controlling the maximum error within $\mathcal{M}_n$.
\end{remark}

Give a mesh $\mathcal{T}_n$, the collection of marked elements is given by
\begin{align}\label{eq:mark}
\mathcal{M}_n:=\cup\left\{
T\in\mathcal{T}_n: \eta(T)\leq \operatorname{tol}_1(T)
\right\}.
\end{align}
Then we apply the iterated longest edge bisection algorithm (LEB) \cite{rivara-1992} to divide each marked element $T\in\mathcal{M}_n$ into two children elements, which typically follows two steps.  First, creating a new vertex at the midpoint of its longest edge and then dividing this tetrahedron by the plane that passes through this new vertex and the two vertices opposite the longest edge. Second, performing additional bisection to avoid any hanging nodes by means of recursive algorithm, which was proven to be finite \cite{rivara-1992}.

We summarize the proposed mesh refinement procedure in Algorithm \ref{alg1}, where the parameters $\operatorname{tol}_2$
and $n_{\max}^1$ denote the smallest allowed element size and the maximum number of iterations, respectively.
\begin{algorithm}[H]
\caption{Adaptive mesh refinement}\label{alg1}
\begin{algorithmic}
\Require initial mesh $\mathcal{T}_1$,
tolerance $\operatorname{tol}_2$, and maximum loop number $n_{\max}^1$
\Ensure $\mathcal{T}_n$
\State $n \gets 1$
\While{$n \leq n_{\max}^1$}

\State Set $\mathcal{M}_n=\emptyset$
\For{$T\in\mathcal{T}_n$}

\If{$\eta(T)\leq \operatorname{tol}_1(T)$ and $h_T\geq \operatorname{tol}_2$}
\State $\mathcal{M}_n=\mathcal{M}_n\bigcup\{T\}$
\EndIf
\EndFor

Refine the elements in $\mathcal{M}_n$ by iterated LEB algorithm and adjust the mesh until there are no hanging nodes.

$n \gets n+1$
\EndWhile
\State $n_{\operatorname{iter}}:=n-1$
\end{algorithmic}
\end{algorithm}
Next, we show that all elements containing singular points are marked for refinement.
For the mesh $\mathcal{T}_n$ at the $n$th iteration, let
\begin{align*}
\mathcal{T}_n^R:=\cup\left\{T\in \mathcal{T}_n:T\cap \mathbf{S}=\emptyset\right\}\text{ and }
\mathcal{T}_n^S:=\mathcal{T}_n\backslash \mathcal{T}_n^R
\end{align*}
be a partition of the mesh $\mathcal{T}_n$ based upon the singular point set $\mathbf{S}$.
\begin{theorem}[$\mathcal{T}_n^S\subset \mathcal{M}_n$]\label{refining mechanism}
 Algorithm \ref{alg1} with the element-wise quantity \eqref{indicator1} and the tolerance \eqref{tolerance} guarantees that all elements containing branch points are marked for refinement.
\end{theorem}
\begin{proof}
The proof follows directly from that of \cite[Theorem 3.2]{wang2023dispersion2}.
\end{proof}
In practice, we use a practical tolerance given by
\begin{align}\label{eq:tol1-num}
\Hat{C}_{\max,\ell}(T):=\max_{\substack{\max\{1,\ell-1\}\leq q \leq \ell+1\\\mathbf{k}\in\mathcal{V}_T}}\left\{\left|\nabla\lambda_q(\mathbf{k})\right|\right\},
\end{align}
as a substitute for $C_{\max,\ell}(T)$ which only requires $\partial_j\lambda_q$ for $i=1,2,3$ and $\max\{1,\ell-1\}\leq q \leq \ell+1$ over $\mathcal{V}_T$, which can be obtained from \eqref{eq:partial-lambda}.

\subsection{Element-wise interpolation over the adaptive mesh}\label{element-wise interpolation}
Let $\{\mathcal{T}_n\}_{n=1}^{n_{\operatorname{iter}}}$ be a sequence of
nested regular tetrahedral meshes generated by Algorithm \ref{alg1}. First, we establish a finite element space of continuous element-wise polynomial functions $V_n$ as the ansatz space to approximate the target band functions. Then, we define element-wise interpolation operator $\Pi_{\mathcal{T}_n}: C(\mathcal{B})\to V_n$ to reconstruct the target band functions on conforming finite element spaces. First, we introduce the concept of a layer of each element $T\in \mathcal{T}_n$.
\begin{definition}\label{layer}
The layer $\ell_T$ of an element $T\in \mathcal{T}_n$ is $\ell_T:=n-r$, where $0\leq r\leq n-1$ is the number of refinements that have been performed on the element $T$ in Algorithm \ref{alg1}.
\end{definition}
Next, we define the total degree for each element, face and edge.
\begin{definition}[Element-wise, face-wise and edge-wise spaces of polynomial functions]\label{degree design strategy}
For each element $T\in\mathcal{T}_n$, we take $P_{n_T}(T)$ as the local interpolation space, with
\begin{equation*}
n_T:=
\left\{
\begin{aligned}
&\max\left\{2,\left\lceil \mu {\ell}_T \right\rceil  \right\} &&\text{ if } T\notin \mathcal{M}_n\\
&2 &&\text{ otherwise}.
\end{aligned}
\right.
\end{equation*}
Here, $\left\lceil \cdot \right\rceil$ denotes the ceiling function, the slope parameter $\mu>0$ will be determined by \eqref{eq:mu-cond} and $\mathcal{M}_n$ is the collection of marked elements \eqref{eq:mark}. Similarly, for each face $F\in\mathcal{F}_n$, let $P_{m_F}(F)$ be the local interpolation space over $F$ with
$m_F:=\min_{T\in\mathcal{T}_n\atop F\in \partial T}\{n_T\}$,
and let $P_{m_E}(E)$ be the local interpolation space over $E$ for $E\in\mathcal{E}_n$ with
$m_E:=\min_{F\in\mathcal{F}_n\atop E\in \partial F}\{m_F\}$.
\end{definition}

This construction implies
\begin{equation*}
		\max_{F\in\mathcal{F}_T}\{m_F\}\leq n_T\quad \mbox{and}\quad 
		\max_{E\in\mathcal{E}_T}\{m_E\}\leq \min_{F\in\mathcal{F}_T\atop E\in \partial F}\{m_F\},\quad \forall T\in \mathcal{T}_n.
\end{equation*}

Finally, we define the conforming element-wise polynomial space. On a given mesh $\mathcal{T}_n$ with $\{n_T\}_{T\in\mathcal{T}_n}$, $\{m_E\}_{E\in\mathcal{E}_n}$, and $\{m_F\}_{F\in\mathcal{F}_n}$ being the element-wise, face-wise, and edge-wise degrees, the global approximation space $V_n$ is given by
\begin{align}\label{def:vn}
V_n:=\left\{v\in C(\mathcal{B}): v|_{T}\in P_{n_T}(T),\text{ } v|_{F}\in P_{m_F}(F),\text{ } v|_{E}\in P_{m_E}(E),\,  \forall T\in\mathcal{T}_n, F\in \mathcal{F}_n,E\in \mathcal{E}_n\right\},
\end{align}
which is a linear space consisting of continuous piecewise polynomial functions on $\mathcal{T}_n$ with predetermined degrees given by the triple $\{n_T, m_F, m_E\}$.

Now, we introduce the local approximation space on the reference triangular pyramid $\hat{T}$:
\begin{equation*}
    \hat{T}:=\{(x,y,z):0\leq x\leq 1, 0\leq y\leq 1, 0\leq z\leq 1, 0\leq 1-x-y-z\leq 1\}.
\end{equation*}
The vertices of $\hat{T}$ are
$\mathbf{\hat{z}}_1=(1,0,0)$, $\mathbf{\hat{z}}_2=(0,1,0)$, $\mathbf{\hat{z}}_3=(0,0,1)$ and $\mathbf{\hat{z}}_4=(0,0,0)$. Let $\hat{F}^i$ be the face opposite to $\mathbf{\hat{z}}_i$ and let $\hat{E}^{jk}$ with $j>k$ be the edge with endpoints $\mathbf{\hat{z}}_j$ and $\mathbf{\hat{z}}_k$. For any $\{n,p_i,q_{jk}\}\in(\mathbb{N}^{+})^{11}$ with $n\geq p_i$ and $p_i\geq q_{jk}$ for $i,j,k=1,2,3,4$, $j,k\neq i$, $j>k$, \eqref{def:vn} implies that
the local space of polynomial functions over $\hat{T}$ is
\begin{align}\label{eq:loc-space}
\hat{P}:=\left\{v\in P_n(\hat{T}): v|_{\hat{F}^i}\in P_{p_i}(\hat{F}^i)\text{ and }v|_{\hat{E}^{jk}}\in P_{q_{jk}}(\hat{E}^{jk}),\text{ for }i,j,k=1,2,3,4,\text{ with } j>k\right\}.
\end{align}
Note that $\hat{P}\subset P_n(\hat{T})$ and $\hat{P}\subsetneq P_n(\hat{T})$ if $p_i<n$ or $q_{jk}<p_i$ for certain $j,k\neq i$ with $j>k$. Note also that the dimension of $P_{n}(\hat{T})$ is $\dim(P_{n}(\hat{T}))=(\substack{n+3\\3})$ and the dimension of $\hat{P}$ is
\begin{equation*}
\dim(\hat{P})=4+\sum_{j=2}^4\sum_{k=1}^{j-1}(q_{jk}-1)+\sum_{i=1}^4\left(\substack{p_i-1\\2}\right)+\left(\substack{n-1\\3}\right).
\end{equation*}
Next we construct a set of basis functions for the local polynomial space $\hat{P}$.
We first partition the degrees of freedom into edge, face, and internal ones according
to whether they vanish over the edge, face, and boundary. The basis functions can
further be classified into four types: nodal, edge, face, and internal ones, following
the standard practice in the community of $hp$-adaptive FEMs \cite{babuvska1994p}. It enables
defining conforming finite element spaces with nonuniform degrees of polynomials over the mesh.

First, the nodal shape functions are defined as the following four polynomials:
\begin{equation}\label{eq:nodal}
    \begin{aligned}
        N^1_1&=x,\quad
        N^1_2=y,\quad
        N^1_3=z,\quad
        N^1_4=1-x-y-z.
    \end{aligned}
\end{equation}
Note that for $i=1,2,3,4$, we have
$N^1_i(\mathbf{\hat z}_i)=1$ and $N^1_i(\mathbf{z})=0$ for $\mathbf{z}\in \hat{F}^i$.

Second, the set of edge shape functions over each edge $\hat{E}^{jk}$ for $j=1,2,3,4$, $j>k$ are
\begin{align}\label{eq:edgeB}
\overline{P}_{q_{jk}}(\hat{E}^{jk}):=
\text{span}\{N^{q_{jk},\nu}_{jk}:=N^1_{j}N^1_{k}\Phi_{\nu-2}(N^1_{k}-N^1_{j}):\nu=2,\cdots,q_{jk}\},
\end{align}
where $\Phi_{\nu-2}(x):=x^{\nu-2}$ is a polynomial of degree $\nu-2$.
The edge functions vanish on all vertices and all faces that do not include the edge.

Third, the face shape functions over each face $\hat{F}^{i}$ for $i=1,2,3,4$ are defined as
\begin{equation}\label{face shape functions}
    \overline{P}_{p_i}(\hat{F}^i):=\left\{b_{\hat{F}^i}\right\}\otimes P_{p_i-3}(\hat{F}^i)\text{ for }p_i\geq 3.
\end{equation}
Here $b_{\hat{F}^i}:=N^1_jN^1_kN^1_l$ with $j\neq k\neq l\neq i$ denotes the face bubble function and $P_{p_i-3}(\hat{F}^i)$ is the space of polynomials defined on $\hat{F}^{i}$ with total degree no greater than $p_i-3$. Moreover, we set $P_{0}(\hat{F}^i)=\{1\}$. Clearly, $\operatorname{supp}(b_{\hat{F}^i})=\hat{F}^i$.

Altogether, the set of external shape functions is given by
\begin{equation*}
    P^{\mathrm{E}}(\partial\hat{T}):=\text{span}\left\{N^1_i:i=1,2,3,4\right\}+\sum_{\substack{j,k=1,\cdots,4\\ j>k}}\overline{P}_{q_{jk}}(\hat{E}^{jk})+\sum_{i=1}^4 P_{p_i}(\hat{F}^i).
\end{equation*}
Let $b_{\hat{T}}:=N^1_1 N^1_2 N^1_3 N^1_4$ be the basic bubble function on the reference element $\hat{T}$ that vanishes over $\partial\hat{T}$. Then the set of internal shape functions $P^I_n(\hat{T})$ are defined as
\begin{equation}\label{internal shape functions}
    P^{\mathrm{I}}_n(\hat{T}):=
    \left\{b_{\hat{T}}\right\}\otimes P_{n-4}(\hat{T})\text{ for }n\geq 4\text{ with }P_{0}(\hat{T})=\{1\}.
\end{equation}

%\subsubsection*{Local interpolation points}
Next, we introduce a set of Lagrangian interpolation points over the reference triangular pyramid $\hat{T}$ with the local polynomial space $\hat{P}$ \eqref{eq:loc-space} such that its restriction to each face and each edge serves as the Lagrangian interpolation points within the face or edge called the Lobatto tetrahedral (LTT) set \cite{luo2006lobatto}. First we introduce a 1D grid in the interval
$[0,1]$ comprising a set of $n+1$ sampling points
$v_i=\tfrac{1}{2}(1+t_i)$,
for $i=2,3,\cdots,n$, where $t_i$ are the zeros of the $(n-1)$th degree Lobatto polynomial. The endpoints are $v_1=0$ and $v_{n+1}=1$.

\begin{definition}[Sampling points over each edge, face and internal]\label{defn:f-e-i-spts}
Suppose the edge degrees on $\hat{E}^{jk}$, face degrees on $\hat{F}^i$ and total degree are $q_{jk}$, $p_i$ and $n$ for $i,j,k=1,\cdots,4$, $j>k$. Then the sampling points are defined below.
\begin{itemize}
\item The set of sampling points on all edges are
\begin{equation}\label{eq:samplingpts-edge}
    \mathcal{S}=\{M_{jk}v_{\nu}:j,k=1,2,3,4, j>k \text{ and }\nu=1\cdots, q_{jk}+1\},
\end{equation}
with $M_{jk}$ being the linear transformation maps $[0,1]$ to $\hat{E}^{jk}$.
\item The sampling points on $\hat{F}^1,\cdots,\hat{F}^4$ are $(0,\zeta,\eta)$, $(\zeta,0,\eta)$, $(\zeta,\eta,0)$, $(\zeta,\eta,\xi)$ with
\begin{equation}\label{eq:samplingpts-face}
\zeta:=\tfrac{1}{3}(1+2v_j-v_k-v_\ell),\quad \eta:=\tfrac{1}{3}(1-v_j+2v_k-v_\ell),\quad\xi:=1-\zeta-\eta,
\end{equation}
for $j=1,2,\cdots,{p}_m$, $k=2,3,\cdots,{p}_m+2-j$ and $\ell={p}_m+3-j-k$ satisfying $\zeta,\eta\neq 0$, $\zeta+\eta\neq 1$ for $\hat{F}^m$, $m=1,2,3$ and $\zeta,\eta,\xi\neq 0$, $\xi+\zeta\neq 1$, $\zeta+\eta\neq 1$, $\xi+\eta\neq 1$, $\xi+\zeta\neq 1$ for $\hat{F}^4$. Here, ${p}_m$ is the face-wise degree on the given face $\hat{F}^m$.
\item The internal sampling points are introduced at the positions $(\xi,\zeta,\eta)$ with
\begin{equation}\label{eq:internal-points}
            \xi=\tfrac{1}{4}(1+3v_i-v_j-v_k-v_\ell),\,\,
            \zeta=\tfrac{1}{4}(1-v_i+3v_j-v_k-v_\ell),\,\,
            \eta=\tfrac{1}{4}(1-v_i-v_j+3v_k-v_\ell),
\end{equation}
        for $i=2,3,\cdots,n$, $j=2,3,\cdots,n+1-i$, $k=2,3,\cdots,n+2-i-j$, and $\ell=n+4-i-j-k$.
    \end{itemize}
\end{definition}

Next, we give an interpolation operator $\Pi: C(\hat{T}) \to \hat{P}$, such that for all $f\in C(\hat{T})$, there holds
\begin{equation}\label{eq:element-interp-property}
	\left\{
	\begin{aligned}
		\Pi f\big|_{\hat{F}^i}&\in P_{p_i}(\hat{F}^i)\quad\qquad\text{ for } i=1,2,3,4,&\\
		\Pi f\big|_{\hat{E}^{jk}}&\in P_{q_{jk}}(\hat{E}^{jk})\qquad\text{ for } j,k=1,2,3,4,\text{ and }j>k,&\\
		\Pi f&\in \hat{P}.		&
	\end{aligned}
	\right.
\end{equation}

First, we introduce the external interpolation $\operatorname{E}: C(\partial\hat{T})\to P^{E}(\partial\hat{T})$. For any function $f\in C(\partial\hat{T})$, the external interpolation $\operatorname{E}f$ is a linear interpolation on a set of sampling points on $\partial\hat{T}$ given by Definition \ref{defn:f-e-i-spts} in the set of external shape functions $P^{\mathrm{E}}(\partial\hat{T})$.
Next, we extend $\operatorname{E}$ from $\partial\hat{T}$ to $\hat{T}$ and denote it as $\mathbf{E}:P^{\mathrm{E}}(\partial\hat{T})\to\hat{P}$.
Then we define the internal interpolation operator $\operatorname{I}: C(\hat{T})\to P^{\mathrm{I}}_n(\hat{T})$ as the Lagrange interpolation on internal sampling point set \eqref{eq:internal-points} using the set of internal shape functions \eqref{internal shape functions}. Denote the Lebesgue constant of $\operatorname{I}$ as $\Upsilon_n$, i.e.,
\begin{align}\label{eq:lebegue-const}
\|\operatorname{I}f\|_{\infty,\hat{T}}\leq \Upsilon_n\|f\|_{\infty,\hat{T}},\quad\forall f\in C(\hat{T}).
\end{align}
It has been shown numerically \cite{luo2006lobatto, chan2015comparison} that $\Upsilon_n$ has moderate growth with respect
to the total degree $n$. Finally, we can define the element-wise interpolation $ \Pi: C(\hat{T})\to \hat{P}$ by
\begin{equation}\label{eq:pi-n}
    \Pi f:=\mathbf{E}(\operatorname{E} f)+\operatorname{I} (f-\mathbf{E}(\operatorname{E} f)) ,\quad\forall f\in C(\hat{T}).
\end{equation}

By construction, $\Pi f$ fulfills the desired properties \eqref{eq:element-interp-property}. Moreover, the following properties hold. (i) $\Pi p=p$, for all $p\in \hat{P}$;
(ii) $\Pi f\big|_{\hat{E}^{jk}}\in P_{q_{jk}}(\hat{E}^{jk})$ is the Lagrange interpolation on the sampling point set \eqref{eq:samplingpts-edge} along each edge $\hat{E}^{jk}$ for $j,k=1,2,3,4$ with $j>k$;
(iii) $\Pi f\big|_{\hat{F}^i}\in P_{p_i}(\hat{F}^i)$ is the Lagrange interpolation on the sampling point set \eqref{eq:samplingpts-face} along each face $\hat{F}^i$ for $i=1,2,3,4$;
(iv) $\Pi f\big|_{\hat{T}^\backslash\partial T}=\operatorname{I}f\in P^I_n(\hat{T})$ is the Lagrange interpolation of $f$ on the sampling point set inside $\hat{T}$ \eqref{eq:internal-points}.

For a given mesh $\mathcal{T}_n$, we define the global approximation operator $\Pi_{\mathcal{T}_n}: C(\mathcal{B})\to V_n$ by applying $\Pi$ to each element $T\in\mathcal{T}_n$ as
\begin{equation}\label{eq:interp-global}
	\Pi_{\mathcal{T}_n}f\big|_{T}:=\left(\Pi(f\circ M_T)\right)\circ M_T^{-1},\quad \forall T\in \mathcal{T}_n,
\end{equation}
where the affine transformation $M_T:=A_T\mathbf{z}+\mathbf{b}$ ($A_T\in\mathbb{R}^{3\times 3}$ is invertible and $\mathbf{b}\in \mathbb{R}^3$) maps the reference tetrahedron $\hat{T}$ to the "physical" element $\Bar{T}$.

Now we can give the following algorithm for approximating $\lambda_{\ell}(\mathbf{k})$ and $\lambda_{\ell+1}(\mathbf{k})$. By sequentially generating the adaptive mesh and the polynomial degree required for each element, each face, and each edge, we can then use the interpolation formula \eqref{eq:interp-global}  to adaptively interpolate $\lambda_q$ for $q=n, n+1$.

\begin{algorithm}[H]
	\caption{$hp$-adaptive sampling algorithm}\label{alg2}
	\begin{algorithmic}
		\Require $\mathcal{T}_{n_{\operatorname{iter}}}$: the outcome of adaptive mesh from Algorithm \ref{alg1};
		
		$\{\ell_T\}_{T\in \mathcal{T}_{n_{\operatorname{iter}}}}$: layer of each element defined in Definition \ref{layer};
		
		$\mu$: a positive parameter defined in Definition \ref{degree design strategy} satisfying \eqref{eq:mu-cond};
		\Ensure $\lambda_q$, $q=\ell,\ell+1$

		Define element-wise, face-wise and edge-wise degrees: $\{n_T\}_{T\in \mathcal{T}_{n_{\operatorname{iter}}}}$, $\{m_F\}_{F\in \mathcal{F}_{n_{\operatorname{iter}}}}$ and $\{m_E\}_{E\in\mathcal{E}_{n_{\operatorname{iter}}}}$  by Definition \ref{degree design strategy};
		
		Compute the local interpolation $\Pi\lambda_q$ by \eqref{eq:pi-n};
		
		Compute the global interpolation $\Pi_{\mathcal{T}_{n_{\operatorname{iter}}}}\lambda_q$ by \eqref{eq:interp-global}.
		
	\end{algorithmic}
\end{algorithm}

\section{Convergence analysis}\label{sec:convergence}
This section is devoted to the convergence analysis of Algorithm \ref{alg2}. We prove an exponential convergence rate in Theorem \ref{exp} when the number of crossings among $\{\lambda_q: \max\{1,\ell-1\}\leq q \leq \ell+2\}$ is finite, and an algebraic convergence in Theorem \ref{algebraic} otherwise. To this end, we first establish the approximation property of the local polynomial interpolation operator $\Pi$ in Theorem \ref{local error theorem}. Then we formulate the element-wise approximation property of $\Pi_{\mathcal{T}_n}$ in Theorems \ref{exp} and \ref{algebraic} under assumption \eqref{eq:mu-cond} on the slope parameter $\mu$ such that the approximation error in the marked element patch $\mathcal{M}_n$ dominates.

\subsection{Local approximation error}
First we establish the approximation property of the local interpolation operator $\Pi$. Recall that for any element-wise, face-wise, and edge-wise degrees $\{n_T\}_{T\in\mathcal{T}_n}$, $\{m_F\}_{F\in\mathcal{F}_n}$, and $\{m_E\}_{E\in\mathcal{E}_n}$, the local space $\hat{P}$ is defined by \eqref{eq:loc-space}. Below the notation $A \lesssim B$ stands for that $A/B$ is bounded above by a constant independent of the total degree, face degree, and edge degree.

First, we establish the stability of the external interpolation $E_n: C(\partial\hat{T})\to P^{E}(\partial\hat{T})$. Although there is no proof, numerical results show that the Lebesgue constant of the set of sampling points on each face $\hat{F}^i$ \eqref{eq:samplingpts-face} is $\mathcal{O}(p_i)$ \cite{blyth2006lobatto}. Hence we make the following assumption
\begin{align}\label{eq:lebesgue-face}
\|\operatorname{E}f\|_{\infty,\hat{F}^i\backslash\partial\hat{F}^i}\lesssim p_i \|f\|_{\infty,\hat{F}^i\backslash\partial\hat{F}^i},\quad \forall  f\in C(\hat{F}^i\backslash\partial\hat{F}^i).
\end{align}

\begin{theorem}\label{theorem lebesgue}
{Let \eqref{eq:lebesgue-face} hold.} Then the Lagrange interpolation $\operatorname{E}|_{\hat{F}^i}$ along each face $\hat{F}^i$ is stable, for $i=1,2,3,4$. Moreover, there holds
	\begin{equation}\label{eq:en}
		\|\operatorname{E} f\|_{\infty,\hat{F}^i}\lesssim p_i\log p_i\|f\|_{\infty,\hat{F}^i}\quad \forall f\in C(\hat{F}^i).
	\end{equation}
\end{theorem}
\begin{proof}
This result follows from \cite[Theorem 5.3]{wang2023dispersion2}.
\end{proof}
\begin{theorem}\label{theorem extension}
	For each $f\in C(\partial \hat{T})$ with $f|_{\hat{F}^i}\in P_{p_i}(\hat{F}^i)$ for all $i=1,2,3,4$ and $f|_{\hat{E}^{jk}}\in P_{q_{jk}}(\hat{E}^{jk})$ for all $j,k=1,2,3,4$ with $j>k$, the extension $\mathbf{E}:P^{\mathrm{E}}(\partial\hat{T})\to \hat{P}$ satisfies
	\begin{equation}\label{eq:ext}
		\|\mathbf{E}(f)\|_{\infty,\hat{T}}\lesssim \|f\|_{\infty,\partial \hat{T}}.
	\end{equation}
	Moreover, the extension map $\mathbf{E}:f\mapsto \mathbf{E}(f)$ is a bounded linear operator.
\end{theorem}
\begin{proof}
First, we assume that $f$ vanishes at all vertices, all edges and all faces but one face. We assume this face is the third face $\hat{F}^3=\{(x,y,z):0\leq x,y,1-x-y\leq 1,z=0\}$, i.e., only face shape functions over $\hat{F}^3$ remain \eqref{face shape functions}. Let $P_{p_3-3}(\hat{F}^3):=\text{span}\{\phi_{\nu}:\nu=1,\cdots,\mu(p_3)\}$ with $\mu(p_3):=(\substack{p_3-1\\2})$, then $\mathbf{E}(f)$ can be written in the form
\begin{equation}\label{eq:xxxxxx}
\mathbf{E}(f)=b_{\hat{F}^3}\sum_{\nu=1}^{\mu(p_3)}c_{\nu}{\phi}_{\nu}.
\end{equation}
Then there holds
\begin{align*}
	&\max_{(x,y,z)\in \hat{T}}|\mathbf{E}(f)|=\max_{(x,y,z)\in \hat{T}}\bigg|b_{\hat{F}^3}\sum_{\nu=1}^{\mu(p_3)} c_{\nu}\phi_{\nu}\bigg|\\
   =&\max_{\substack{0\leq x\leq 1,0\leq y\leq 1,0\leq z\leq 1\\x+y+z\leq 1}}\bigg|xy(1-x-y-z)\cdot\sum_{\substack{0\leq i,j\leq p_3-3\\0\leq i+j\leq p_3-3}} c_{ij}x^iy^j\bigg|.
\end{align*}
Consequently, straightforward calculation leads to	
\begin{align}
			\max_{(x,y,z)\in \hat{T}}|\mathbf{E}(f)|&= \max_{\substack{0\leq x\leq 1,0\leq y\leq 1,z=0\\x+y\leq 1}}\bigg|xy(1-x-y-z)\cdot\sum_{\substack{0\leq i,j\leq p_3-3\\0\leq i+j\leq p_3-3}} c_{ijk}x^iy^j\bigg|\nonumber\\
			%&= \max_{(x,y,z)\in \partial\hat{T}}\bigg|b_{\hat{F}^3}\cdot \sum_{\nu=1}^{\mu(p_3)} c_{\nu}\phi_{\nu}\bigg|
			=\max_{(x,y,z)\in \partial\hat{T}}|f|.\label{eq:face-ext}
\end{align}
Second, we assume that $f$ vanishes at all faces, edges and vertices but one edge. We assume the edge to be $\hat{E}^{21}$. Then $\mathbf{E}(p)$ can be written in the form
	\begin{equation*}
		\begin{aligned}
			\mathbf{E}(f)=N^1_{2}N^1_{1}\sum_{\nu=2}^{q_{21}}c_{\nu}(N^1_{2}-N^1_{1})^{\nu-2}=xy\sum_{\nu=2}^{q_{21}}c_{\nu}(y-x)^{\nu-2}.
		\end{aligned}
	\end{equation*}
	This leads to
		\begin{align}
			\max_{(x,y)\in \hat{T}}|\mathbf{E}(f)|&=\max_{(x,y)\in \hat{T}}\left|xy\sum_{\nu=2}^{q_{21}}c_{\nu}(y-x)^{\nu-2}\right|\nonumber\\
			&=\max_{\substack{0\leq x\leq 1,0\leq y\leq 1,0\leq z\leq 1\\x+y+z\leq 1}}\left|xy\sum_{\nu=2}^{q_{21}}c_{\nu}(y-x)^{\nu-2}\right|\nonumber\\
			&=\max_{\substack{0\leq x\leq 1,0\leq y\leq 1,z=0\\x+y= 1}}\left|xy\sum_{\nu=2}^{q_{21}}c_{\nu}(y-x)^{\nu-2}\right|
			=\max_{(x,y)\in \partial\hat{T}}|f|.\label{eq:edge-ext}
		\end{align}
Third, we assume that $f$ vanishes at all faces, edges and vertices, but one vertex. Similarly, we have
\begin{equation}\label{eq:node-ext}
		\begin{aligned}
			\max_{(x,y)\in \hat{T}}|\mathbf{E}(f)|=\max_{(x,y)\in \partial\hat{T}}|f|.
  \end{aligned}
	\end{equation}
Finally, we can prove the general case of $f\in C(\partial \hat{T})$ with $f|_{\hat{F}^i}\in P_{p_i}(\hat{F}^i)$ for all $i=1,2,3,4$ and $f|_{\hat{E}^{jk}}\in P_{q_{jk}}(\hat{E}^{jk})$ for all $j,k=1,2,3,4$ with $j>k$. The linearity of the extension operator $\mathbf{E}$, together with \eqref{eq:face-ext}, \eqref{eq:edge-ext} and \eqref{eq:node-ext}, implies
	\begin{align*}
    &\max_{(x,y,z)\in \hat{T}}|\mathbf{E}(f)|=\max_{(x,y,z)\in \hat{T}}\bigg|\sum_{i=1}^4\mathbf{E}(f\big|_{\hat{F}^i})+\sum_{\substack{j,k=1\\j>k}}^4\mathbf{E}(f\big|_{\hat{E}^{jk}}) +\sum_{\ell=1}^4N_{\ell}^1 f(\hat{\mathbf{z}}_{\ell})\bigg|  \\	
    \leq&\sum_{i=1}^4\max_{(x,y,z)\in \hat{T}}|\mathbf{E}(f\big|_{\hat{F}^i})| +\sum_{\substack{j,k=1\\j>k}}^4\max_{(x,y)\in \hat{T}}|\mathbf{E}(f\big|_{\hat{E}^{jk}})|
   +\sum_{\ell=1}^4\max_{(x,y)\in \hat{T}}|N_{\ell}^1f(\hat{\mathbf{z}}_{\ell})|\\
	= & \sum_{i=1}^4\max_{(x,y,z)\in \partial\hat{T}}|\mathbf{E}(f\big|_{\hat{F}^i})| +\sum_{\substack{j,k=1\\j>k}}^4\max_{(x,y)\in \partial\hat{T}}|\mathbf{E}(f\big|_{\hat{E}^{jk}})|
   +\sum_{\ell=1}^4\max_{(x,y)\in \partial\hat{T}}|N_{\ell}^1f(\hat{\mathbf{z}}_{\ell})|\lesssim\Vert f\Vert_{\infty,\partial \hat{T}},
	\end{align*}
where $f\big|_{\hat{F}^i}$ denotes the face polynomials interpolating the sampling points over $\hat{F}^i$, $i=1,\cdots,4$; $f\big|_{\hat{E}^{jk}}$ are edge polynomials interpolating the sampling points over one edge, $j,k=1,\cdots,4$, $j>k$; $\sum_{\ell=1}^4N_{\ell}^1f(\hat{\mathbf{z}}_{\ell})$ are nodal shape functions interpolating $\hat{\mathbf{z}}_{\ell}$, $\ell=1,\cdots,4$. This proves the desired result.
\end{proof}
Finally, we can state the following quasi-optimal approximation property of $\Pi$ \eqref{eq:pi-n}.

\begin{theorem}[Quasi-optimal approximation property of $\Pi$]\label{local error theorem}
	For every $n\in\mathbb{N}^+$, the interpolation operator $\Pi:C(\hat{T})\to \hat{P}$ satisfies
	\begin{equation*}
		\|f-\Pi f\|_{\infty,\hat{T}}\lesssim \cstab{n}\inf_{p\in \hat{P}}\|f-p\|_{\infty,\hat{T}}  \text{ with } \cstab{n}:=\Upsilon_n n\log n.\
	\end{equation*}
\end{theorem}
\begin{proof}
	The stability of $\operatorname{E}$ and $\mathbf{E}$ in \eqref{eq:en} and \eqref{eq:ext} gives the stability of $\mathbf{E}\operatorname{E}: C(\partial \hat{T})\to \hat{P}$:
	\begin{equation}\label{eq:extension-stab}
		\|\mathbf{E}(\operatorname{E} f)\|_{\infty,\hat{T}}\lesssim\|\operatorname{E} f\|_{\infty,\partial \hat{T}}\lesssim n\log n\|f\|_{\infty,\partial \hat{T}}.
	\end{equation}
	Together with the stability estimate of the internal interpolation operator $\operatorname{I}$ in  \eqref{eq:lebegue-const}, we derive
	\begin{equation*}
		\begin{aligned}
			\|\Pi f\|_{\infty,\hat{T}}&\lesssim \|\mathbf{E}(\operatorname{E} f)\|_{\infty,\hat{T}}+\|\operatorname{I}(f-\mathbf{E}(\operatorname{E}f))\|_{\infty,\hat{T}} \\
			&\lesssim n\log n\|f\|_{\infty,\partial \hat{T}}+\Upsilon_n\|f-\mathbf{E}(\operatorname{E}f)\|_{\infty,\hat{T}} \\
			&\lesssim n\log n\|f\|_{\infty,\partial \hat{T}}+\Upsilon_n\left(\|f\|_{\infty,\hat{T}}
			+\|\mathbf{E}(\operatorname{E} f)\|_{\infty,\hat{T}} \right).
		\end{aligned}
	\end{equation*}
In view of \eqref{eq:extension-stab}, we obtain
\begin{equation*}
			\|\Pi f\|_{\infty,\hat{T}}\lesssim
			n\log n\|f\|_{\infty,\partial \hat{T}}+\Upsilon_n \|f\|_{\infty,\hat{T}}+\Upsilon_n n\log n\|f\|_{\infty,\partial\hat{T}}\lesssim \Upsilon_n n\log n\|f\|_{\infty,\hat{T}}.
\end{equation*}
Note that $\Pi \phi=\phi$ for all $\phi\in \hat{P}$. Consequently, we arrive at
\begin{equation*}
\|f-\Pi \phi\|_{\infty,\hat{T}}=\|(f-\phi)-\Pi(f-\phi)\|_{\infty,\hat{T}}\lesssim \Upsilon_n n\log n\|f-\phi\|_{\infty,\hat{T}},\quad \forall \phi\in\hat{P}.
\end{equation*}
Taking the infimum over all $\phi\in \hat{P}$ leads to the desired assertion.
\end{proof}

\begin{theorem}[Element-wise approximation property of $\Pi_{\mathcal{T}_n}$]
	\label{property}
	Let $n\in\mathbb{N}$ be sufficiently large.
	Let $\mathcal{T}_n$ with element-wise, face-wise, and edge-wise degrees $\{n_T\}_{T\in\mathcal{T}_n}$, $\{m_F\}_{F\in\mathcal{F}_n}$ and $\{m_E\}_{E\in\mathcal{E}_n}$ be generated by Algorithm \ref{alg1}, Definitions \ref{degree design strategy} and let the conforming element-wise polynomial space be given by \eqref{def:vn}. Then for $q=\ell,\ell+1$, there holds
	\begin{align}
		\Vert \lambda_q-\Pi_{\mathcal{T}_n}\lambda_q\Vert_{\infty, T}&\lesssim   \mathrm{C}_{\rm bpa}h_T |\lambda_q|_{W^{1,\infty}(T)} &&\text{for }T\in\mathcal{M}_n,\label{error t2}\\
		\Vert \lambda_q-\Pi_{\mathcal{T}_n}\lambda_q\Vert_{\infty, T}&\lesssim \cstab{n_T} \mathrm{C}_{\rm bpa}C(\Tilde{n}) \left(\frac{\mathrm{C}_{\rm const}h_T}{\Tilde{n}}\right)^{\Tilde{n}}|\lambda_q|_{W^{\tilde{n},\infty}(T)}&&\text{for }T\in\mathcal{T}_n\backslash\mathcal{M}_n,\label{error t1}
	\end{align}
	where $\mathrm{C}_{\rm bpa}$ is a bounded constant independent of $q$, $\Tilde{n}$ and $h_{\hat{T}}$, $\Tilde{n}:=\min_{E\in\mathcal{E}_T}\{m_{E}\}+1$, $\mathrm{C}_{\rm const}:=\frac{h_{\hat{T}}}{\rho_{\hat{T}}}=\frac{3+\sqrt{3}}{\sqrt{2}}$, and $|\cdot|_{W^{k,\infty}(D)}$ is the $W^{k,\infty}(D)$ semi-norm.
Furthermore, if the slope parameter $\mu$ satisfies the following condition for all $T\in\mathcal{T}_n\backslash\mathcal{M}_n$
	\begin{equation}\label{eq:mu-cond}
		\ln\left(\cstab{\mu \ell_T} C(n_T+1)C_{\lambda_q,T}^{\left\lceil \mu \ell_T+1 \right\rceil}\right)+\left\lceil \mu \ell_T+1 \right\rceil \left(\ln \mathrm{C}_{\rm const}+F(\ell_T)-\ln\left\lceil \mu \ell_T+1 \right\rceil\right)
		\lesssim \ln\left(C_{\max}^q\right)+F(1),
	\end{equation}
	 then there holds
	\begin{align*}
		\max_{T\in \mathcal{T}_n\backslash\mathcal{M}_n}\Vert \lambda_q-\Pi_{\mathcal{T}_n}\lambda_q\Vert_{\infty, T}\lesssim
		\min_{T\in\mathcal{M}_n}\Vert\lambda_q-\Pi_{\mathcal{T}_n}\lambda_q\Vert_{\infty, T}.
	\end{align*}
	Here, $C_{\max}^q:=\max_{T\in\mathcal{M}_n}  |\lambda_q|_{W^{1,\infty}(T)}$, $C_{\lambda_q,T}^{\left\lceil \mu \ell_T+1 \right\rceil}:=|\lambda_q|_{W^{\left\lceil \mu \ell_T+1 \right\rceil,\infty}(T)}$, and $F(\ell):=\ln h_1-\frac{n-\ell}{3}\ln 2$ with $h_1$ being the initial mesh size.
\end{theorem}
\begin{proof}
	The proof is similar to Theorem 5.4 in \cite{wang2023dispersion2}, but with $\frac{h_{\hat{T}}}{\rho_{\hat{T}}}=\frac{3+\sqrt{3}}{\sqrt{2}}$ and $h_T\sim h_12^{-(n-\ell_T)/3}$.
\end{proof}	
\subsection{Global approximation error}
Now assuming that the slope parameter $\mu$ in Definition \ref{degree design strategy} satisfies \eqref{eq:mu-cond}, we present the convergence rates of Algorithm \ref{alg2} when the number of crossings among $\{\lambda_q: \max\{1,\ell-1\}\leq q \leq \ell+2\}$ 
 is finite and infinite. We denote $N$ to be the number of sampling points in Algorithm \ref{alg2}. The proofs to Theorems \ref{exp} and \ref{algebraic} follow in a similar manner to that of \cite[Theorems 5.5 and 5.6]{wang2023dispersion2}.

\begin{theorem}[Convergence rate for Algorithm \ref{alg2} with finite branch points]\label{exp}
Let $n_{\operatorname{iter}}\in\mathbb{N}$ be sufficiently large.
Let $\mathcal{T}_{n_{\operatorname{iter}}}$ with element-wise, face-wise, and edge-wise degrees $\{n_T\}_{T\in\mathcal{T}_n}$, $\{m_F\}_{F\in\mathcal{F}_n}$, and $\{m_E\}_{E\in\mathcal{E}_n}$ be generated by Algorithm \ref{alg1}, Definition \ref{degree design strategy}, the conforming element-wise polynomial space be defined in \eqref{def:vn}. If the slope parameter $\mu$ satisfies \eqref{eq:mu-cond} and the number of crossing among $\{\lambda_q: \max\{1,\ell-1\}\leq q \leq \ell+2\}$ is finite, then there holds
\begin{equation*}
\Vert \lambda_q-\Pi_{\mathcal{T}_{n_{\operatorname{iter}}}}\lambda_q\Vert_{\infty,\mathcal{B}}\lesssim \exp(-b N^{\frac{1}{4}})| \lambda|_{W^{1,\infty}(\mathcal{B})}\qquad \text{ for } q=\ell \text{ and }\ell+1.
\end{equation*}
Here, the constant $b:=\frac{1}{3}\ln 2$. 
\end{theorem}
\begin{proof}
Note that when the number of iterations $n_{\operatorname{iter}}$ is sufficiently large, 
\begin{align}\label{eq:3333}
h_T\sim h_1 2^{-n_{\operatorname{iter}}/3} \text{ for all } T\in \mathcal{M}_{n_{\operatorname{iter}}}, \text{ and }N\sim {n_{\operatorname{iter}}}^4
\end{align}
if the number of branch points is finite. Combining with \eqref{eq:mu-cond} and \eqref{error t2}, we derive 
\begin{align*}
\Vert \lambda_q-\Pi_{\mathcal{T}_{n_{\operatorname{iter}}}}\lambda_q\Vert_{\infty,\mathcal{B}}
&\lesssim \min_{T\in\mathcal{M}_{\operatorname{iter}}}\Vert\lambda_q\Pi_{\mathcal{T}_{\operatorname{iter}}}\lambda_q\Vert_{\infty, T}\\
&\lesssim\mathrm{C}_{\rm bpa}h_T |\lambda_q|_{W^{1,\infty}(T)}\\
&\lesssim\mathrm{C}_{\rm bpa}h_1 (2^{1/3})^{-N^{\frac{1}{4}}} |\lambda_q|_{W^{1,\infty}(T)}.
\end{align*}
Here, we have used \eqref{eq:3333} in the last step, and this proves the desired assertion.
\end{proof}

\begin{theorem}[Convergence rate for Algorithm \ref{alg2} with infinite branch points]\label{algebraic}
Under the conditions of Theorem \ref{exp}, if the number of crossing among $\{\lambda_q: \max\{1,\ell-1\}\leq q \leq \ell+2\}$ is infinite, then
\begin{equation*}
\Vert \lambda_q-\Pi_{\mathcal{T}_{n_{\operatorname{iter}}}}\lambda_q\Vert_{\infty,\mathcal{B}}\lesssim N^{-1}| \lambda_j|_{W^{1,\infty}(\mathcal{B})}\qquad \text{ for } q=\ell \text{ and }\ell+1.
\end{equation*}
\end{theorem}
\begin{proof}
Note that when the number of iterations $n_{\operatorname{iter}}$ is sufficiently large, 
\begin{align*}
h_T\sim h_1 2^{-n_{\operatorname{iter}}/3} \text{ for all } T\in \mathcal{M}_{n_{\operatorname{iter}}}, \text{ and }N\sim 2^{n_{\operatorname{iter}}}
\end{align*}
if the number of branch points is infinite. Then a similar argument as in the proof to Theorem \ref{exp} leads to the desired assertion.
\end{proof}

\begin{remark}\label{uniform}
When uniform refinement and fixed low-degree element-wise interpolation are utilized, the number of elements is approximately $\sim 2^{n_{\operatorname{iter}}}$, which indicates a convergence rate of $\mathcal{O}(N^{-1/3})$. Furthermore, the decay rate remains only $\mathcal{O}(N^{-1/3})$ for the global polynomial interpolation method proposed in \cite{wang2023dispersion}. Hence, Algorithm \ref{alg2} achieves at least triple convergence rate compared with the uniform refinement with the fixed low-degree interpolation (degree=2 in our numerical experiment), as well as the global polynomial interpolation method.
\end{remark}

%\begin{remark}
%As discussed in \cite[Remark 5.2]{wang2023dispersion2}, we assume $\mu=1$ is sufficient for $\mathcal{M}_n$ dominating the error. In our numerical results, we will take this value as an example.
%\end{remark}

\section{Numerical experiments and discussions}\label{Section optimization}
Now we apply Algorithm \ref{alg2} to the design of 3D PhCs with wide band gaps. We take
two cubic unit cells $\Omega=[0,a]^3$ with lattice constant $a\in\mathbb{R}^{+}$ filled with air embedded with six silicon blocks and a central silicon sphere. The length, width and height of the blocks, and radius of the sphere are design variables. Although the assumption of a symmetric unit cell facilitates computation, a symmetry reduction of the unit cell can open wider band gaps or create new band gaps \cite{gazonas2006genetic,dong2014topology,dong2015reducing}. Thus, adding more design variables to describe an asymmetrical unit cell has great potential to maximize photonic band gaps. Algorithm \ref{alg2} allows greatly reducing the associated high computational complexity. We utilize \cite{Chen:2008ifem} to implement the iterative LEB algorithm involved in Algorithm \ref{alg1}.

\subsection{Experimental setting}

First we describe two experimental settings, denoted as Model 1 and Model 2 below.

\begin{figure}[hbt!]
\centering
\subfigure[Model 1: unit cell]{\label{fig1}
\includegraphics[width=.3\textwidth,trim={8cm 4cm 9cm 0.5cm},clip]{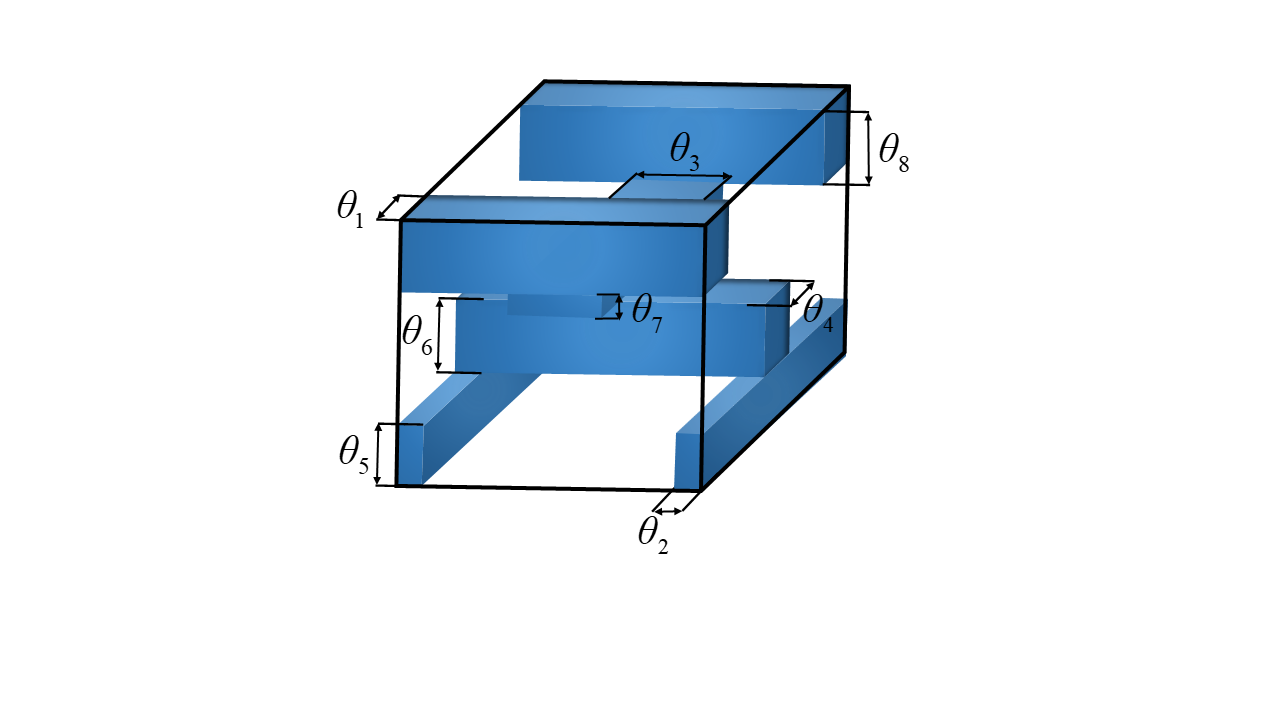}
}%
\quad
\subfigure[Model 1: first Brillouin zone]{\label{fig2}
\includegraphics[width=.27\textwidth,trim={1.2cm 0.7cm 0.5cm 0.5cm},clip]{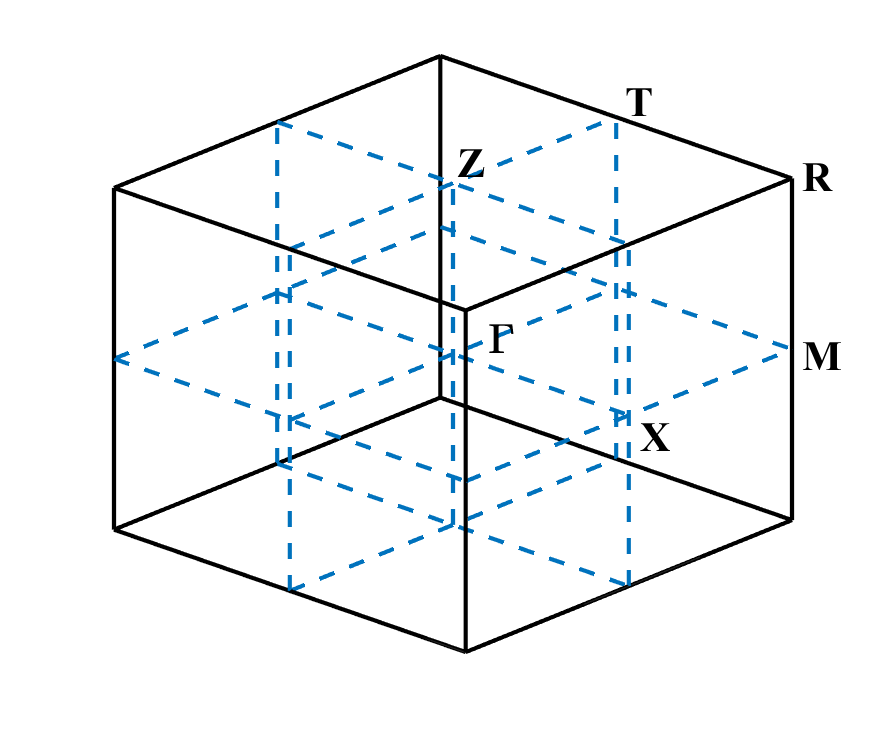}
}%
\\
\subfigure[Model 2: unit cell]{\label{fig1_2}
	\includegraphics[width=.3\textwidth,trim={4cm 3cm 10cm 0.5cm},clip]{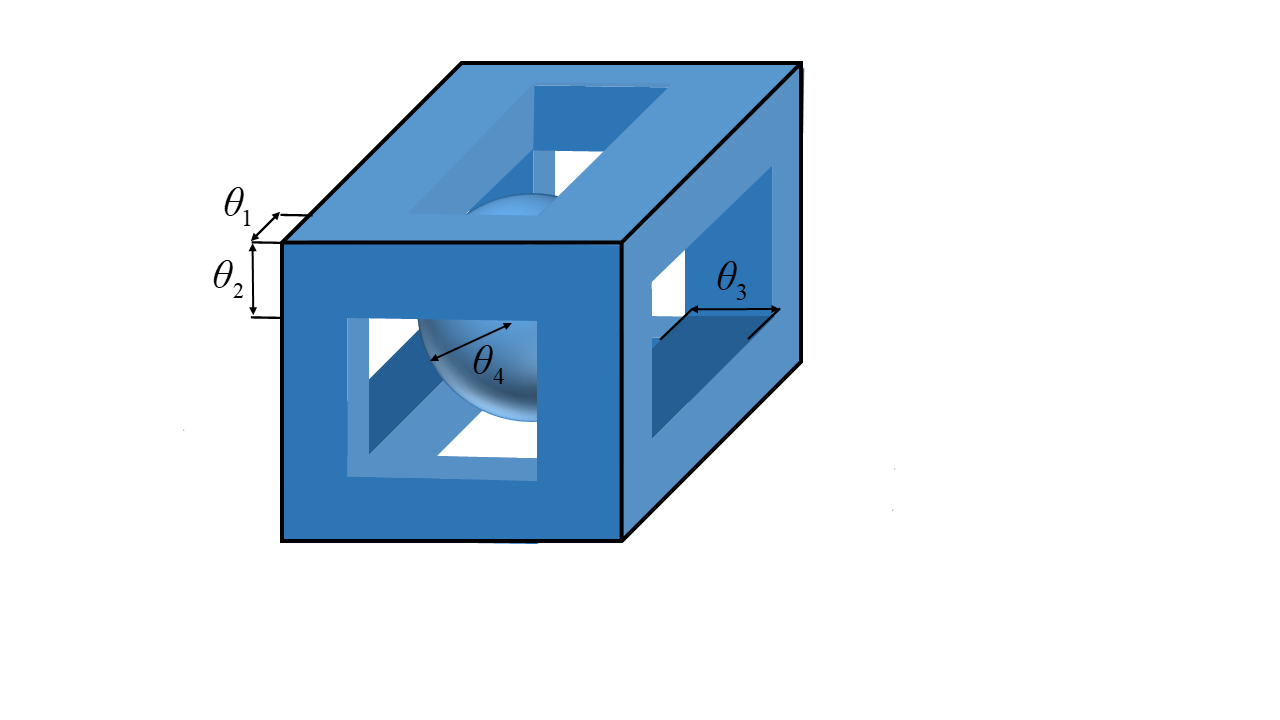}
}%
\quad
\subfigure[Model 2: first Brillouin zone]{\label{fig2_2}
	\includegraphics[width=.27\textwidth,trim={1.2cm 0.7cm 0.5cm 0.5cm},clip]{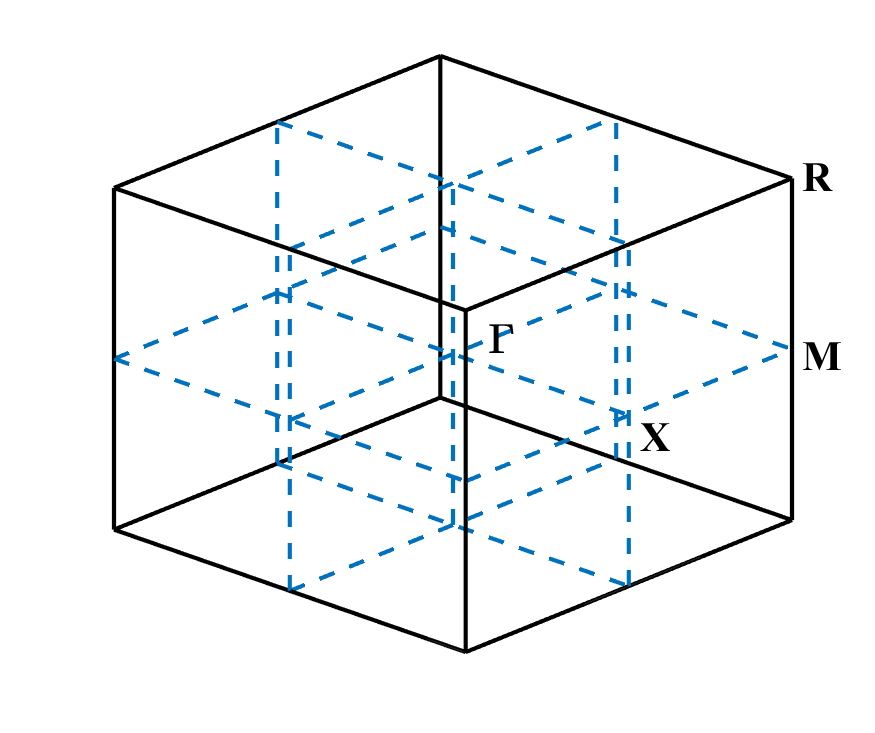}
}%
\centering
\caption{Unit cells and their corresponding first Brillouin zones for Model 1 and Model 2. }
\label{unit cell}
\end{figure}

\begin{model}
There are six silicon blocks layered inside a cubic unit cell $\Omega=[0,a]^3$, cf. Fig. \ref{fig1} \cite{lin1998three}. The relative permittivity $\epsilon$ inside the silicon blocks is $\epsilon=13$. Their arrangement leads to $xy$-plane symmetry within the unit cell, but there is no symmetry along the $z$-axis. The first Brillouin zone $\mathcal{B}$ is shown in Fig. \ref{fig2}, where the high-symmetry points are denoted as $\Gamma=\frac{1}{a}(0,0,0)^T, X=\frac{1}{a}(\pi,0,0)^T, M=\frac{1}{a}(\pi,\pi,0)^T, R=\frac{1}{a}(\pi,\pi,\pi)^T, T=\frac{1}{a}(0,\pi,\pi)^T, Z=\frac{1}{a}(0,0,\pi)^T$. The IBZ is the polyhedron formed by connecting these high-symmetry points. This type of structure has been successfully applied in preclinical studies \cite{mavciulaitis2015preclinical}.
There are eight design parameters $\{\theta_i\}_{i=1}^8$, cf. Fig. \ref{fig1}. We fix the block thickness as $\theta_5=\theta_6=\theta_7=\theta_8=0.25a$ and only seek the optimal parameters $\{\theta_i\}_{i=1}^4$ to maximize the band gap over the admissible set
\begin{align*}
\Theta_{\operatorname{ad},1}:=\left\{(\theta_1,\theta_2,\theta_3,\theta_4): \theta_1,\theta_2\in[0,0.5a]\text{ and }\theta_3,\theta_4\in[0,a]\right\}.
\end{align*}
It is known  that there is a band gap between 4th and 5th bands when $\theta_1=\theta_2=0.125a$ and $\theta_3=\theta_4=0.25a$ \cite{bulovyatov2010parallel}. Note that the commonly used path to characterize the band function $\Gamma\to X\to M\to R\to T\to Z\to\Gamma\to M$ is insufficient to accurately detect band gaps. We will adaptively generate sampling points in the IBZ by Algorithm \ref{alg1} and approximate band functions element-wisely by Algorithm \ref{alg2}.
\end{model}

\begin{model} The unit cell also contains silicon blocks, cf. Fig. \ref{fig1_2}. Additionally, there is a silicon sphere with a radius of $\theta_4$ at the center of the unit cell. This structure imparts symmetry to the unit cell along the $x,  y,$ and $z$ axes. Therefore, the high-symmetry points in $\mathcal{B}$ are $\Gamma=\frac{1}{a}(0,0,0)^T$, $X=\frac{1}{a}(\pi,0,0)^T$, $M=\frac{1}{a}(\pi,\pi,0)^T$ and $R=\frac{1}{a}(\pi,\pi,\pi)^T$, and the region enclosed by the lines connecting them represents IBZ. This photonic crystal structure was fabricated in \cite{liu20223d} by a sacrificial-scaffold mediated two-photon lithography strategy and using a degradable hydrogel scaffold.
There are four design parameters $\{\theta_i\}_{i=1}^4$, which are the dimensions (length and width) of the silicon blocks and the radius of the silicon sphere. We take the following admissible set
\begin{align*}
\Theta_{\operatorname{ad},2}:=\big\{(\theta_1,\theta_2,\theta_3,\theta_4)\in[0,0.5a]^4: (\theta_i^2+\theta_j^2)^{1/2}+\theta_4\leq \sqrt{2}a/2\text{ for }i,j=1,2,3 \text{ and } i\neq j\big\}
\end{align*}
to ensure that there is no overlap between the silicon sphere and the silicon blocks. It can be verified experimentally that there is a band gap when $\theta_1=\theta_2=\theta_3=\frac{1}{6}a$ and $\theta_4=\frac{11}{30}a$.
\end{model}

Next we formulate the band gap maximization problem as PDE constrained optimization. The objective is to determine the unit cell that maximizes the band gap, i.e., the distance between two adjacent band functions $\omega_{\ell}$ and $\omega_{\ell+1}$. One intuitive way is to define the objective function by
\begin{equation}\label{optimization 1}
\hat{\phi}_1(\Bar{\mathbf{\theta}};\ell)=\min_{\mathbf{k}\in\mathcal{B}}\omega_{\ell+1}(\mathbf{k};\Bar{\mathbf{\theta}})-\max_{\mathbf{k}\in\mathcal{B}}\omega_{\ell}(\mathbf{k};\Bar{\mathbf{\theta}}),
\end{equation}
with the design variable vector $\Bar{\mathbf{\theta}}:=[\theta_1,\theta_2,\theta_3,\theta_4]^T$ for shape-wise optimization. Usually, the objective function is normalized by the mean value in which the band gap is turned
\begin{equation}
\Bar{\omega}_{\ell}=\Big(\min_{\mathbf{k}\in\mathcal{B}}\omega_{\ell+1}(\mathbf{k};\Bar{\mathbf{\theta}})+\max_{\mathbf{k} \in\mathcal{B}}\omega_{\ell}(\mathbf{k};\Bar{\mathbf{\theta}})\Big)/2,
\end{equation}
which results in a modified objective function $\hat{\phi}_2(\Bar{\mathbf{\theta}};\ell)=\hat{\phi}_1(\Bar{\mathbf{\theta}};\ell)/\Bar{\omega}_{\ell}$. We normalized the objective function to indicate how the band gap improves relative to the mean band gap frequency since it is crucial to know not only the band gap width but also its location.

Since each eigenfrequency is obtained by taking the square root of the eigenvalue in equation \eqref{strong}, we use the following objective function
\begin{equation*}%\label{objective 1}
\phi(\Bar{\mathbf{\theta}};\ell)=\frac{\min_{\mathbf{k}\in\mathcal{B}}\lambda_{\ell+1,h}(\mathbf{k};\Bar{\mathbf{\theta}})-\max_{\mathbf{k}\in\mathcal{B}}\lambda_{\ell,h}(\mathbf{k};\Bar{\mathbf{\theta}}))}{\left(\min_{\mathbf{k}\in\mathcal{B}}\lambda_{\ell+1,h}(\mathbf{k};\Bar{\mathbf{\theta}})+\max_{\mathbf{k}\in\mathcal{B}}\lambda_{\ell,h}(\mathbf{k};\Bar{\mathbf{\theta}})\right)/2}.
\end{equation*}
In sum, the optimization problem can then be formulated as
 \begin{align}\label{objective 1}
\max_{\Bar{\mathbf{\theta}}\in\Theta_{\operatorname{ad},i}}\phi(\Bar{\mathbf{\theta}};\ell)
\text{ for }i=1,2
 \end{align}
 subject to the following variational formulation with $j=\ell, \ell+1$,
\begin{equation*}
\left\{\begin{aligned}
a(\mathbf{u}_{j,h},\mathbf{v}; \mathbf{k})+b(p_{j,h},\mathbf{v};\mathbf{k})&=\lambda_{j,h} (\mathbf{u}_{j,h},\mathbf{v}),\quad\forall \mathbf{v} \in \mathbf{V}_h^\mathbf{k},\\
\overline{b(q,\mathbf{u}_{j,h};\mathbf{k})}&=0,\qquad\qquad\qquad\forall q \in Q_h^\mathbf{k}.
\end{aligned}\right.
\end{equation*}

\begin{remark}\label{another obj remark}
Typically, the unit cell chosen for optimization already possesses some band gaps under certain predefined parameters. For instance, given a specific set of parameters $\Bar{\mathbf{\theta}}$, there exists a band gap between the $4$th and $5$th band functions for Model 1, i.e., $\phi(\Bar{\mathbf{\theta}};4)>0$. Then we aim to solve \eqref{objective 1} with $\ell:=4$. On the contrary, if the unit cell being optimized does not possess any a priori band gaps, we can use the following objective function,
\begin{equation}\label{objective 2}
\phi(\Bar{\mathbf{\theta}})=\max_{n=1,\cdots,L-1}\frac{\min_{\mathbf{k}\in\mathcal{B}}\lambda_{n+1}(\mathbf{k};\Bar{\mathbf{\theta}})-\max_{\mathbf{k}\in\mathcal{B}}\lambda_{n}(\mathbf{k};\Bar{\mathbf{\theta}})}{\left(\min_{\mathbf{k}\in\mathcal{B}}\lambda_{n+1}(\mathbf{k};\Bar{\mathbf{\theta}})+\max_{\mathbf{k}\in\mathcal{B}}\lambda_{n}(\mathbf{k};\Bar{\mathbf{\theta}})\right)/2}.
\end{equation}
It aims to find an optimal parameter $\hat{\mathbf{\theta}}$ that maximizes the band gap among the first $L$ band functions. By maximizing the difference between the minimum and maximum values of adjacent band functions within the Brillouin zone $\mathcal{B}$, we can effectively enhance the band gap therebetween.
\end{remark}

Next we describe how to solving Problem \eqref{objective 1}, which is divided into two steps. The first step is to evaluate the objective function $\phi(\Bar{\theta}_i;\ell)$ for a given parameter $\Bar{\theta}_i$, which involves calculating $\lambda_{\ell,h}$ and $\lambda_{\ell+1,h}$ by Algorithm \ref{alg2}. The second step is to update the parameter from $\Bar{\theta}_i$ to $\Bar{\theta}_{i+1}$ to obtain a larger value in the objective function $\phi(\Bar{\theta}_{i+1};\ell)$. One repeats these two steps until a certain criterion is attained.
Note that solving \eqref{objective 1} is challenging since it is non-convex and the derivatives of the objective function are neither symbolically nor numerically available. Thus, we adopt some derivative-free algorithms to deal with this problem \eqref{objective 1}. Bayesian optimization (BO) is a popular method for solving optimization problems involving expensive objective functions. Roughly speaking, BO uses a computationally cheap surrogate model to construct a posterior distribution over the original expensive objective $\phi(\Bar{\mathbf{\theta}};\ell)$ using a limited number of observations. The statistics of the surrogate model are used to sample a new
point where the original $\phi(\Bar{\mathbf{\theta}};\ell)$ is then
evaluated. The incorporation of such a surrogate model
makes BO an efficient global optimization technique in terms of the number of function evaluations \cite{mockus1994application,jones1998efficient,streltsov1999non,sasena2002flexibility}. To the best of our knowledge, the application of BO to photonic band gap maximization remains largely unexplored. In this work, we employ MATLAB's built-in optimization solver \textit{bayesopt}. We summarize the procedure in Algorithm \ref{alg optimization}. For more details about the BO algorithm, we refer to  \cite{jones1998efficient,shahriari2015taking}.

\begin{algorithm}[hbt!]
\caption{Bayesian optimization procedure}\label{alg optimization}
\begin{algorithmic}
\Require $D_1=\left\{\left(\mathbf{\Bar{\theta}}_{1},\phi(\mathbf{\Bar{\theta}}_{1};\ell)\right)\right\}$: initial guess $\mathbf{\Bar{\theta}}_{1}$ and its corresponding observation $\phi(\mathbf{\Bar{\theta}}_{1};\ell)$ approximated by Algorithm \ref{alg2};

$\theta^i_{\min},\theta^i_{max}$, $i=1,2,\cdots4$: minimum and maximum values of the design variables;

$n_{\max}^2$: maximum number of loops.
\Ensure $\mathbf{\Bar{\theta}}$
\State $i \gets 1$
\While{$i \leq n_{\max}^2$ and $\phi(\mathbf{\Bar{\theta}}_{i};\ell)>\phi(\mathbf{\Bar{\theta}}_{1};\ell)$}

select new $\mathbf{\Bar{\theta}}_{i+1}$ by optimization of an acquisition function constructed by $D_i$ (using \textit{bayesopt});

get new observation $\phi(\mathbf{\Bar{\theta}}_{i+1};\ell)$ by Algorithm \ref{alg2} with $n_{\max}^1=8$ in Algorithm \ref{alg1};

augment data: 
$D_{i+1}=\left\{D_i,\left(\mathbf{\Bar{\theta}}_{i+1},\phi(\mathbf{\Bar{\theta}}_{i+1};\ell)\right)\right\}$;

update model;

$i \gets i+1$;
\EndWhile
\end{algorithmic}
\end{algorithm}

\subsection{Maximization of band gap}\label{subsec:max-gap}

Now we solve \eqref{objective 1} with $i=1$, $\ell=4$ for Model 1, and $i=2$, $\ell=2$ for Model 2 by Algorithm \ref{alg2}. We use the lowest order $\mathbf{k}$-modified N\'{e}d\'{e}lec edge elements \eqref{discrete variational 2} and discretize the computational domain $\Omega$ with a uniform mesh comprised of 24576 tetrahedral elements in Fig. \ref{mesh 1} for Model 1, 24650 tetrahedral elements in Fig. \ref{mesh 2} for Model 2. The employed mesh serves dual purposes: first, to fulfill periodic boundary conditions, ensuring a one-to-one correspondence of edges along opposing faces of the cubic domain; second, for Model 2, the tetrahedral mesh necessitates partitioning along the spherical material situated at the center of the cubic domain.

\begin{figure}[hbt!]
\centering
\subfigure[Model 1]{\label{mesh 1}
\includegraphics[width=.35\textwidth,trim={1.2cm 0.7cm 0.5cm 0.5cm},clip]{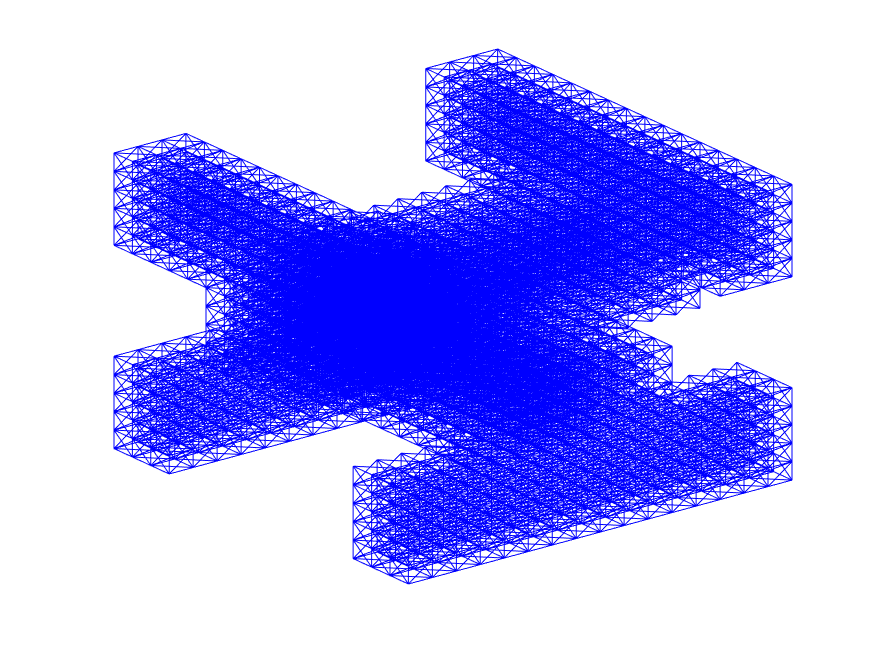}
}%
\quad
\subfigure[Model 2]{\label{mesh 2}
\includegraphics[width=.35\textwidth,trim={1.2cm 0.7cm 0.5cm 0.5cm},clip]{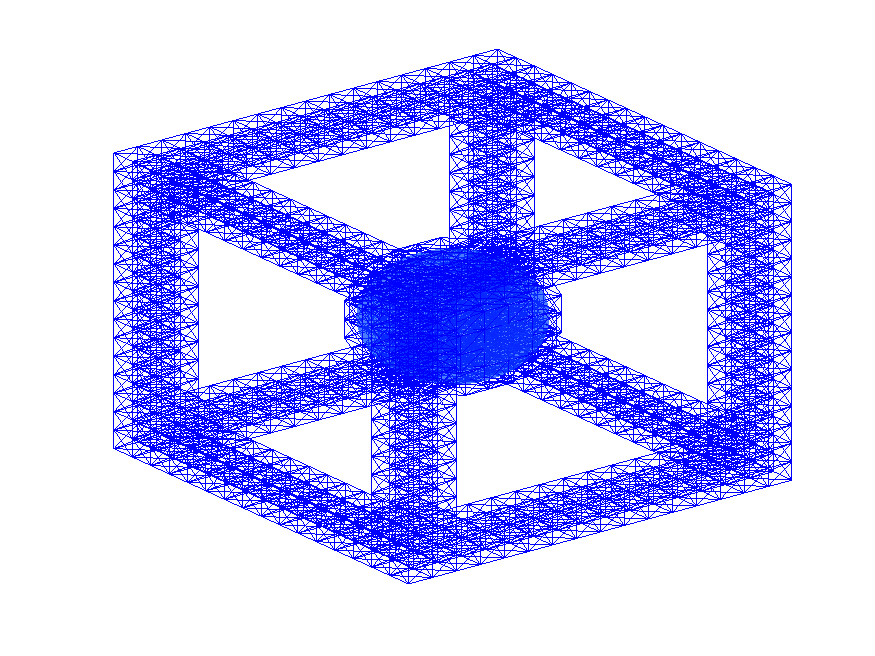}
}%
\centering
\caption{Uniform discretization of the unit cell $\Omega$. Here, only the mesh inside the silicon blocks and around the interior ball is depicted, and the same mesh size is utilized in the background.}
\label{mesh}
\end{figure}

Now we describe the initialization strategy for Algorithm \ref{alg optimization}. \\
(i): {Model 1}. The band structure of the cubic lattice with lattice constant $a$, and design variables $\theta_1=\theta_2=0.125a, \theta_3=\theta_4=\theta_5=\theta_6=\theta_7=\theta_8=0.25a$ is initially computed. The resulting band structure shows a band gap between the $4$th and $5$th bands with normalized band gap width $0.2159$. The first six band functions along the high symmetry points are shown in Fig. \ref{initial band 1}. Although there is already a band gap in the initial structure, band gap optimization facilitates a deeper understanding of how their sizes can affect the band gap via parameter optimization.\\
(ii): {Model 2}. In case $\theta_1=\theta_2=\theta_3=\frac{1}{6}a$, $\theta_4=\frac{11}{30}a$, PhCs with such a unit cell can forbid certain frequencies of light. We consider optimizing these four parameters with silicon blocks and sphere inside the unit cell to find a nearly optimal band gap range. The band structure with initial configuration is computed and the resulting band structure shows a band gap between the $2$nd and $3$rd bands with normalized band gap width $0.05626$. The first six band functions along the high symmetry points are shown in Fig. \ref{initial band 2}.

\begin{figure}[hbt!]
	\centering
	\subfigure[Model 1 initial band]{\label{initial band 1}
		\includegraphics[width=.4\textwidth,trim={0cm 0cm 0.5cm 0.5cm},clip]{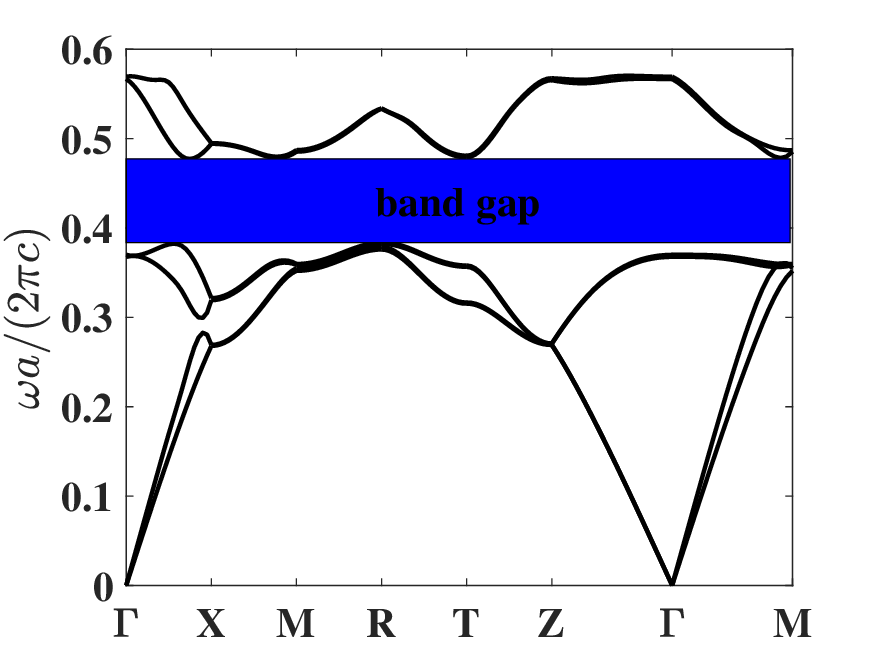}}%
	\quad
	\subfigure[Model 2 initial band]{\label{initial band 2}
		\includegraphics[width=.4\textwidth,trim={0cm 0cm 0.5cm 0.5cm},clip]{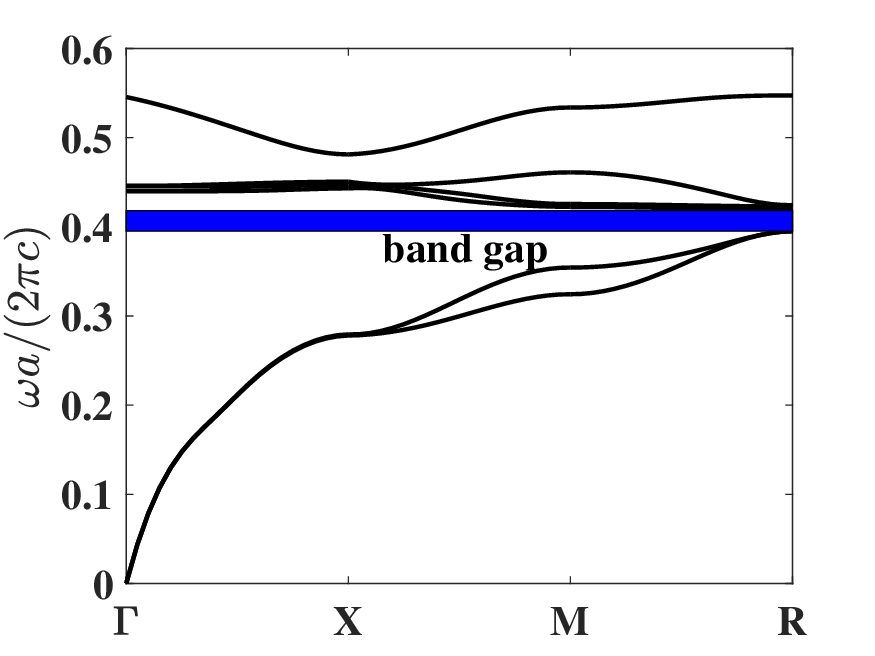}	}%
	\centering
	\caption{Band structures along the high symmetry points.}
	\label{band structure initial}
\end{figure}

Next we describe the optimization results of normalized band gap \eqref{objective 1}.

For {Model 1}, we first fix $\theta_3=\theta_4=0.25a$ and only optimize $(\theta_1,\theta_2)$, which denotes the width of the four blocks on the 1st and 4th layers. The outcome is shown in Fig. \ref{bayes 1}. Several iterations of Algorithm \ref{alg optimization} generate a set of parameters with their band gaps being larger than the known band gap. The optimal values we obtain is $(\theta_1,\theta_2)=(0.1271a,0.1329a)$ and the maximal normalized band gap is $\hat{\phi}_2((0.1271a,0.1329a);4)=0.2168$. This optimization process generates a larger band gap compared to existing literature where the design variables are chosen to be $\theta_1=\theta_2=0.125a$ and the corresponding normalized band gap is $0.2159$. For the optimized parameter, the bandgap width increases by  $0.4177\%$. Further, we optimize parameters $(\theta_1, \theta_2, \theta_3, \theta_4)$ simultaneously. The results show that when $(\theta_1, \theta_2, \theta_3, \theta_4)=(0.1696a,0.1810a,0.2342a,0.3214a)$, one obtains a unit cell design with a larger band gap than the one in the literature. The band gap of the configuration is $0.222$, which is $3.0236\%$ higher than the original one. The first six band functions, under these two new sets of design parameters, along the high symmetry points are shown in Fig. \ref{band structure}.

\begin{figure}[hbt!]
	\centering
	\subfigure[Model 1: optimize $\theta_1,\theta_2$]{\label{bayes 1}
		\includegraphics[width=.45\textwidth,trim={0.5cm 0.1cm 0.5cm 0cm},clip]{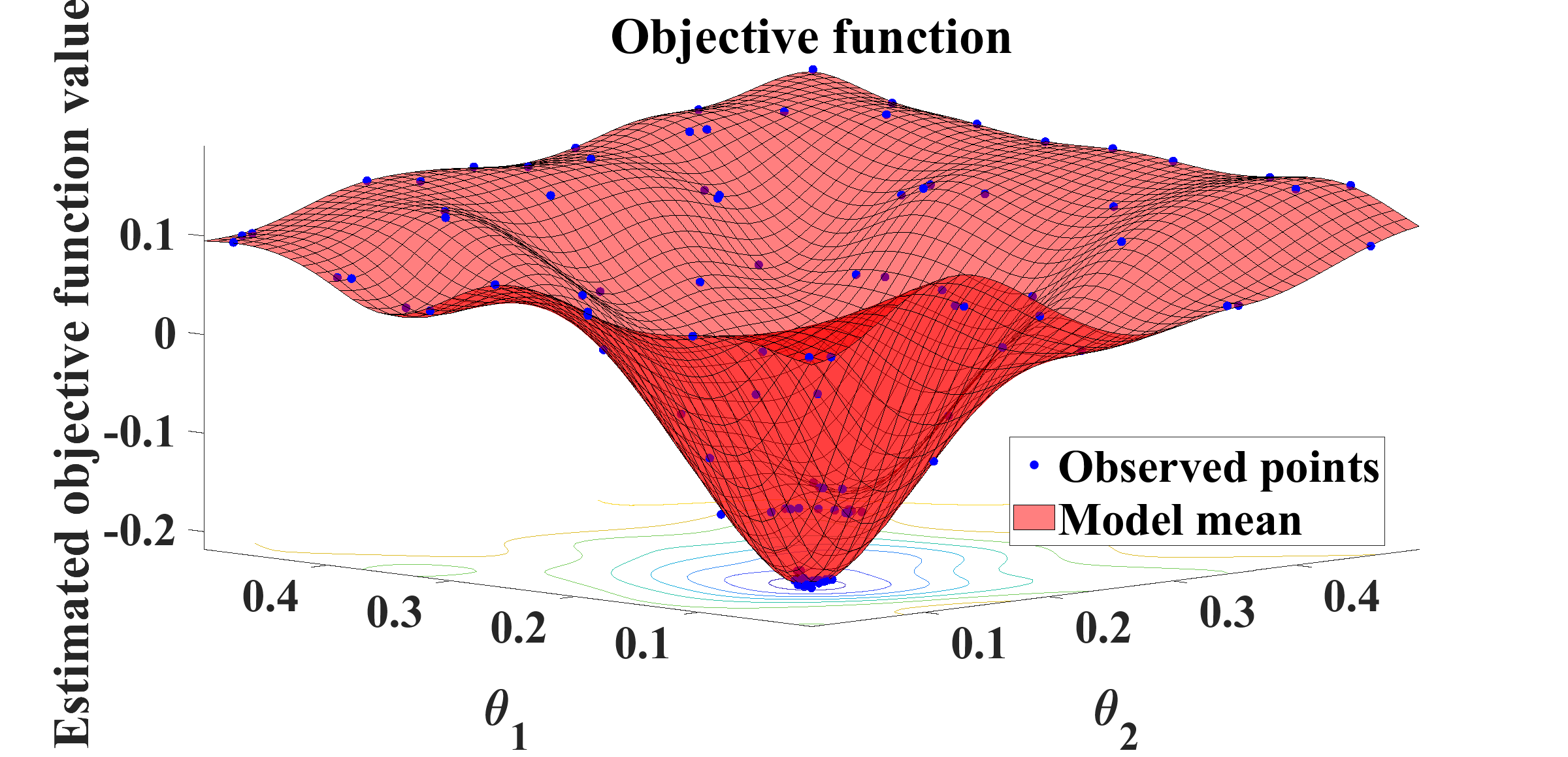}
	}%
	\quad
	\subfigure[Model 2 optimize $\theta_1,\cdots,\theta_4$ with $\theta_1=\theta_2=\theta_3$]{\label{bayes 2}
		\includegraphics[width=.45\textwidth,trim={0.5cm 0.1cm 0.5cm 0cm},clip]{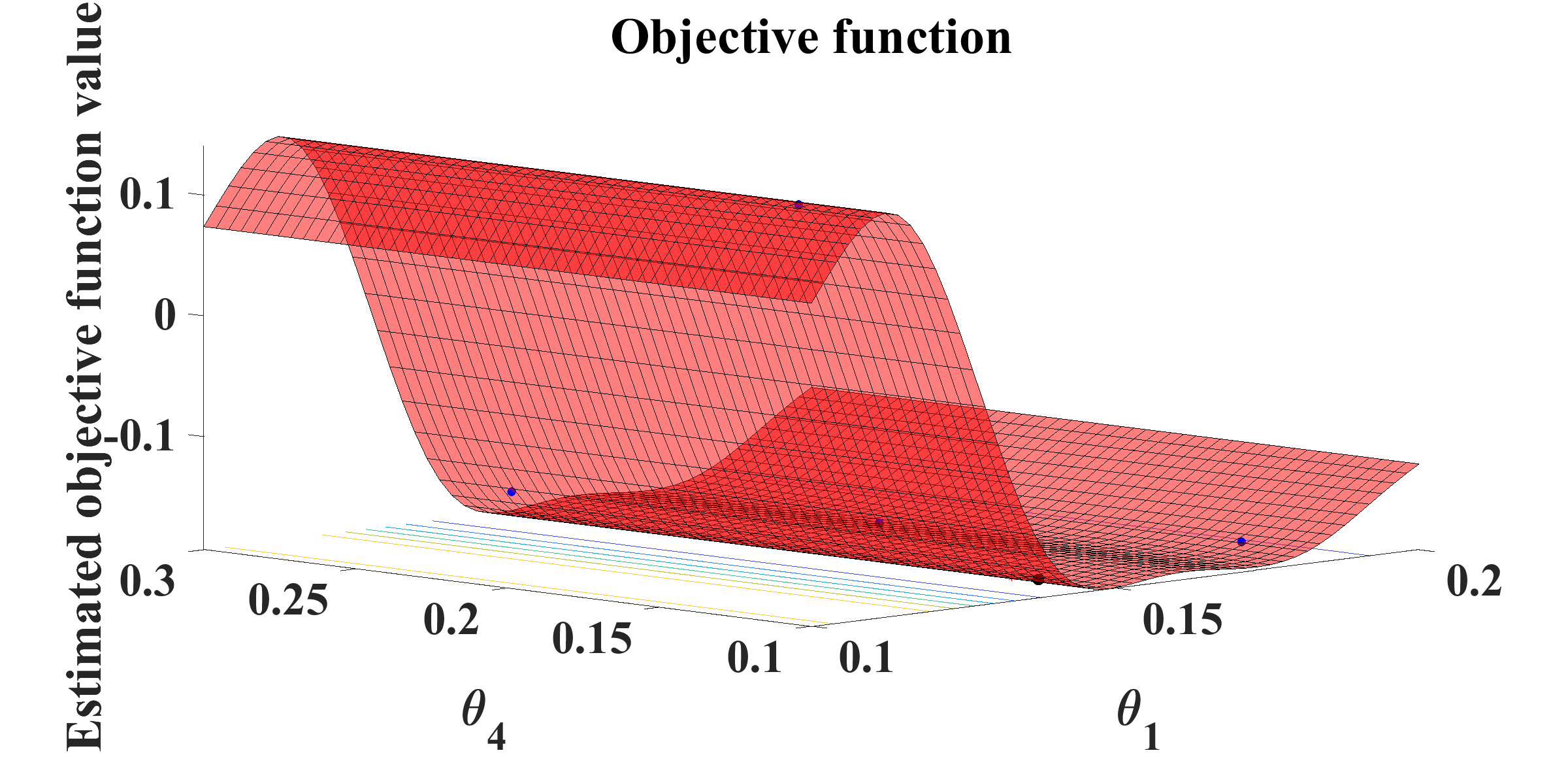}
	}%
	\centering
	\caption{Bayesian optimization.}
	\label{bayes}
\end{figure}

For {Model 2}, we first consider blocks of square cross-section, i.e. $\theta_1=\theta_2=\theta_3$ in \ref{fig1_2} to get a visualized optimization process in Figure \ref{bayes 2}. In this case, Algorithm \ref{alg optimization} produces a set of parameters such that the band gap is larger than that of the configuration in reference \cite{liu20223d} after a few iterations of parameter selection. The optimal parameters are $(\theta_1,\theta_2,\theta_3,\theta_4)=(0.1707a,0.1707a,0.1707a,0.2190a)$ with the normalized band gap is $0.1906$, which is $238.8028\%$ higher than the original band gap. Next we optimize parameters $(\theta_1, \theta_2, \theta_3, \theta_4)$ independently. The results show that $(\theta_1, \theta_2,\theta_3,\theta_4):=0.1798a,0.1777a, 0.1842a,0.2992a)$ yields a unit cell design with a much larger band gap than the one reported in the literature. The band gap of the configuration is $0.2046$, which is $263.6872\%$ higher than the original band gap. The first six band functions along the high symmetry points under these two sets of design parameters are depicted in Fig. \ref{band structure 2}.

\begin{figure}[hbt!]
\centering
\subfigure[band structure when optimizing $\theta_1,\theta_2$.]{\label{band 1}
\includegraphics[width=.4\textwidth,trim={0cm 0cm 0.5cm 0.5cm},clip]{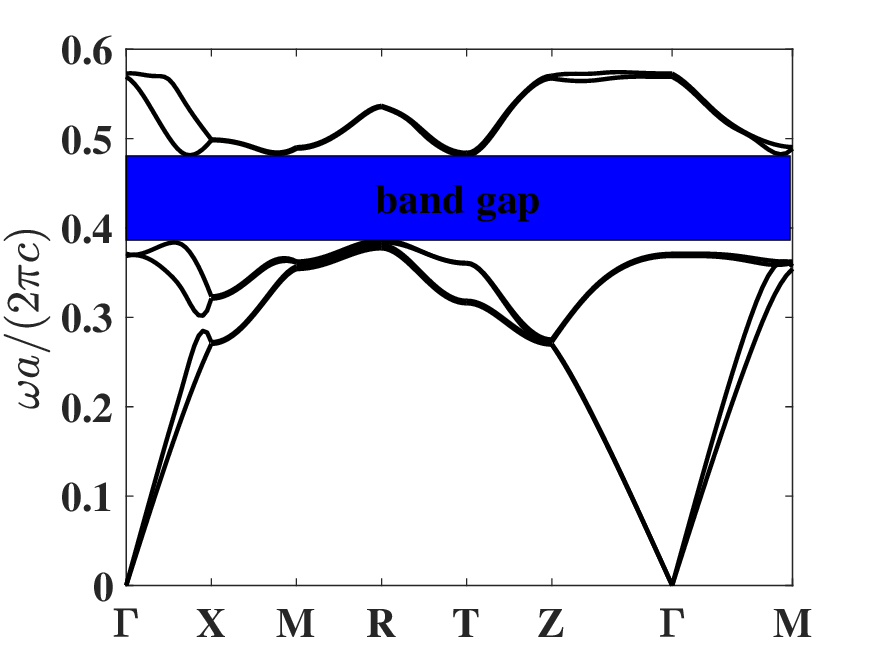}
}%
\quad
\subfigure[band structure when optimizing $\theta_1,\cdots,\theta_4$.]{\label{band 2}
\includegraphics[width=.4\textwidth,trim={0cm 0cm 0.5cm 0.5cm},clip]{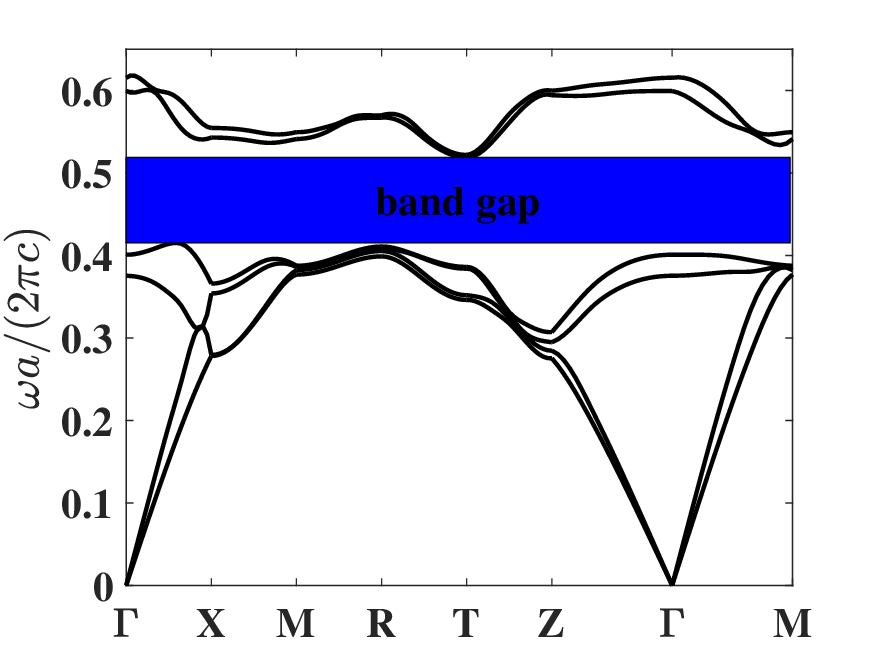}
}%
\centering
\caption{Model 1: Band structures along the high symmetry points.}
\label{band structure}
\end{figure}

During the optimization process, for any generated "next point", we adaptively select the sampling points in the first Brillouin zone and element-wisely interpolate $\lambda_{\ell}(\mathbf{k})$ and $\lambda_{\ell+1}(\mathbf{k})$. During these four optimization processes, we performed 8 loops of adaptive mesh refinement under each set of parameters.  Fig. \ref{points 1} shows the points used to approximate the objective function under the optimal parameters. We observe that the distribution of the sampling points is not uniform.

\begin{figure}[hbt!]
	\centering
	\subfigure[band structure when optimizing $\theta_1,\cdots,\theta_4$ with $\theta_1=\theta_2=\theta_3$.]{\label{band 1_2}
		\includegraphics[width=.4\textwidth,trim={0cm 0cm 0.5cm 0.5cm},clip]{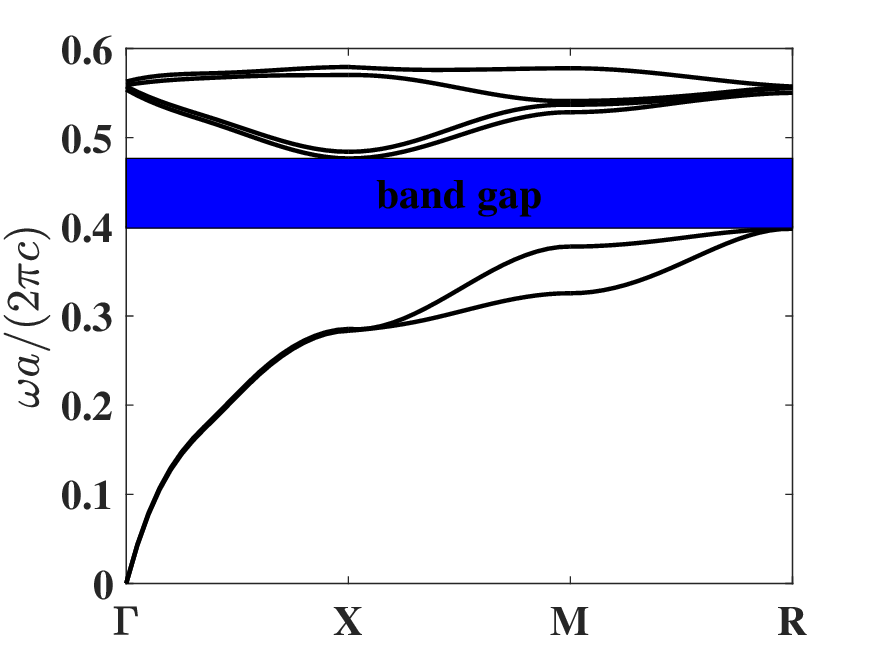}
	}%
	\quad
	\subfigure[band structure when optimizing $\theta_1,\cdots,\theta_4$ independently.]{\label{band 2_2}
		\includegraphics[width=.4\textwidth,trim={0cm 0cm 0.5cm 0.5cm},clip]{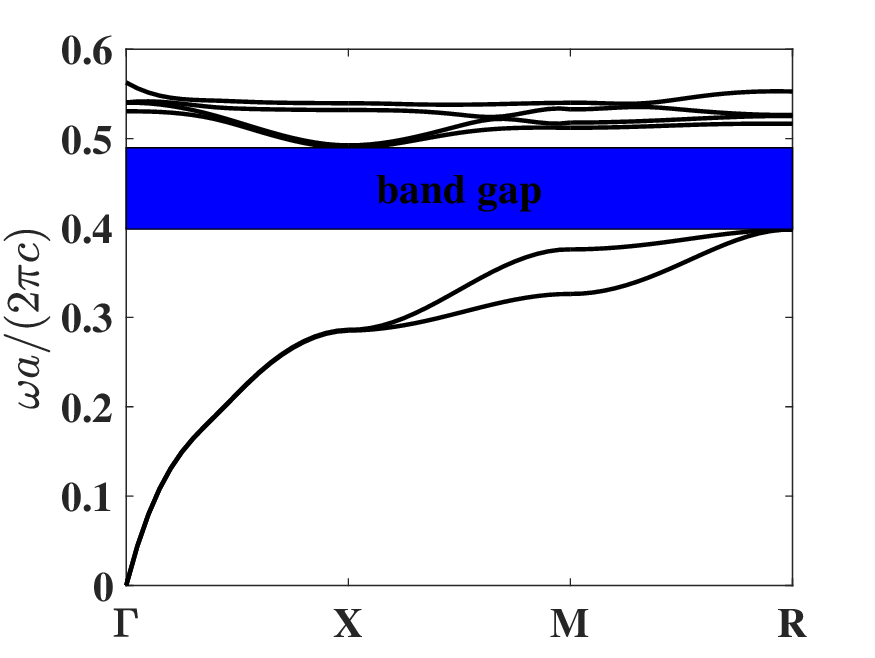}
	}%
	\centering
	\caption{Model 2: Band structures along the high symmetry points.}
	\label{band structure 2}
\end{figure}

\begin{figure}[hbt!]
\centering
\subfigure[Loop 3]{\label{l3}
	\includegraphics[width=.3\textwidth,trim={0cm 0cm 0cm 0cm},clip]{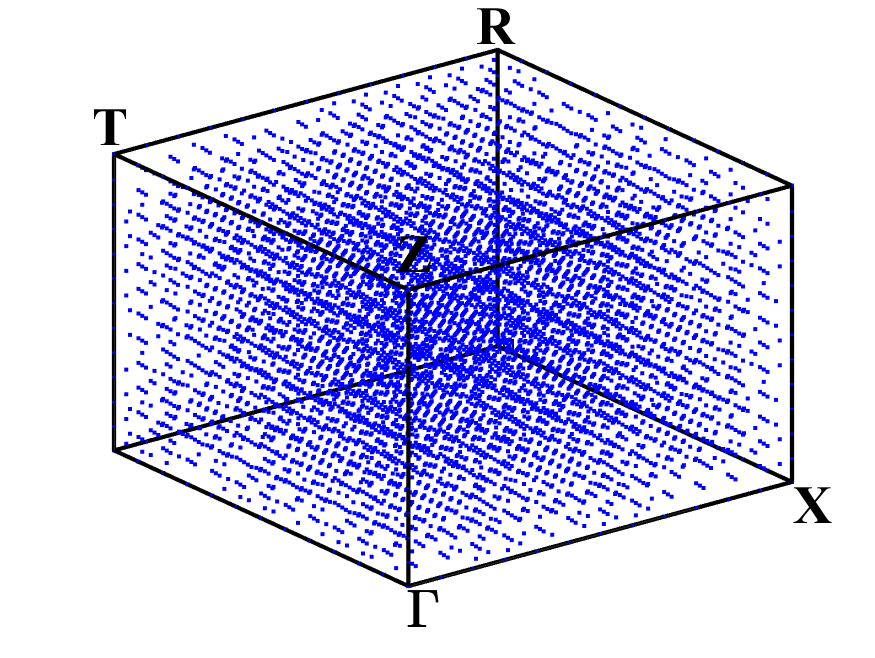}
}%
\quad
\subfigure[Loop 4]{\label{l4}
	\includegraphics[width=.3\textwidth,trim={0cm 0cm 0cm 0cm},clip]{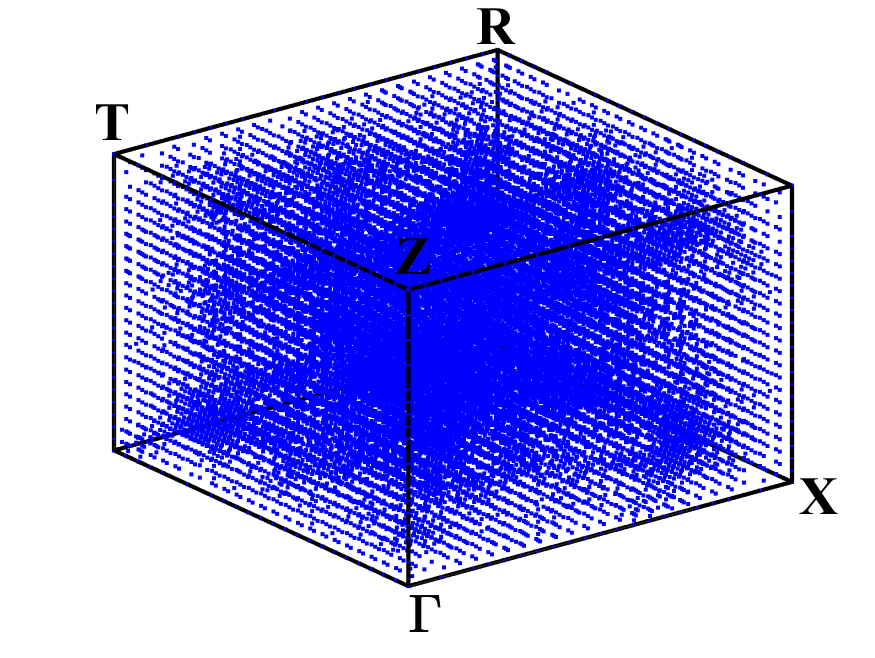}
}%
\quad
\subfigure[Loop 5]{\label{distance l4}
	\includegraphics[width=.3\textwidth,trim={0cm 0cm 0cm 0cm},clip]{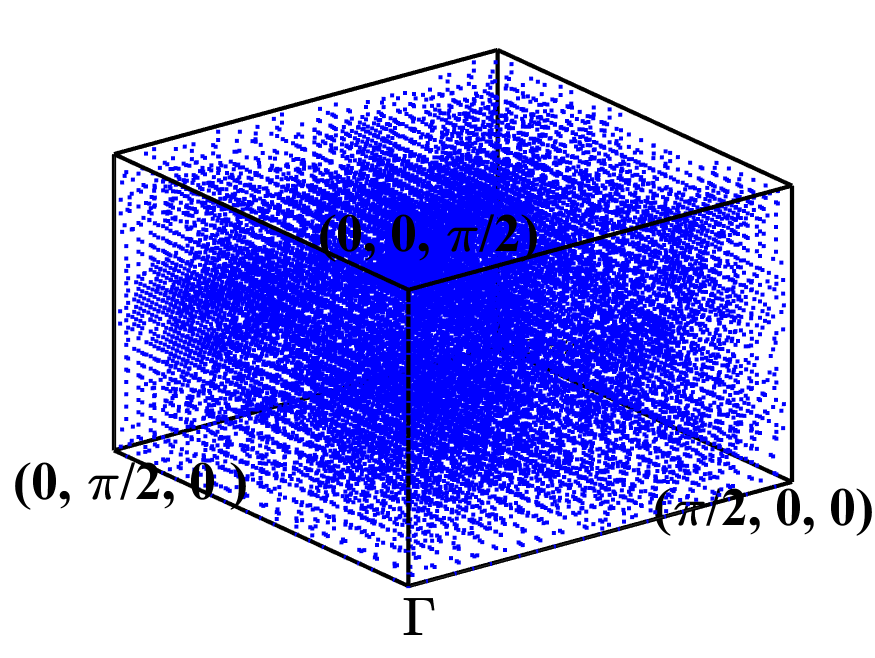}
}%
\quad
\centering
\subfigure[Loop 6]{\label{l5}
	\includegraphics[width=.3\textwidth,trim={0cm 0cm 0cm 0cm},clip]{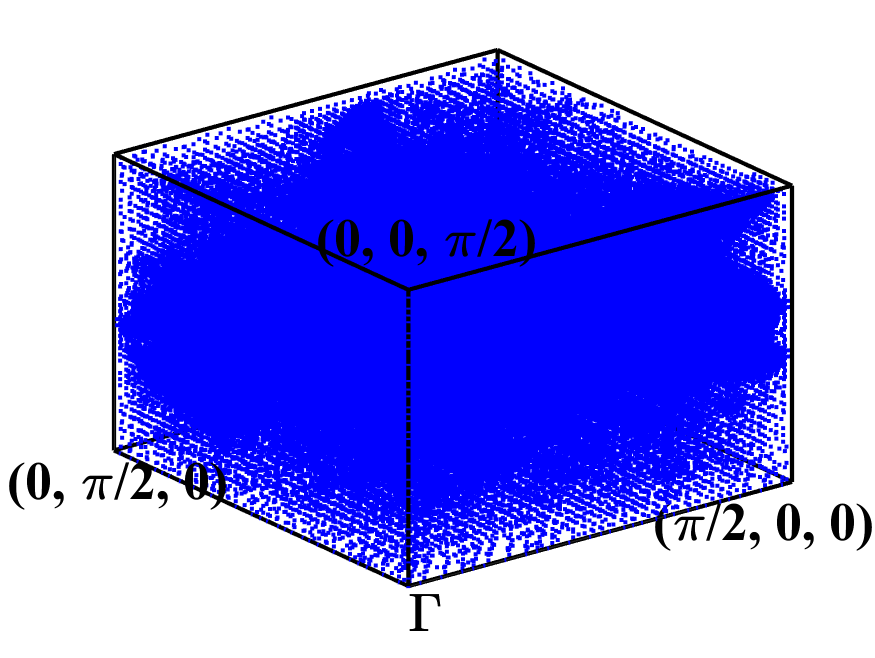}
}%
\quad
\subfigure[Loop 7]{\label{distance l5}
	\includegraphics[width=.3\textwidth,trim={0cm 0cm 0cm 0cm},clip]{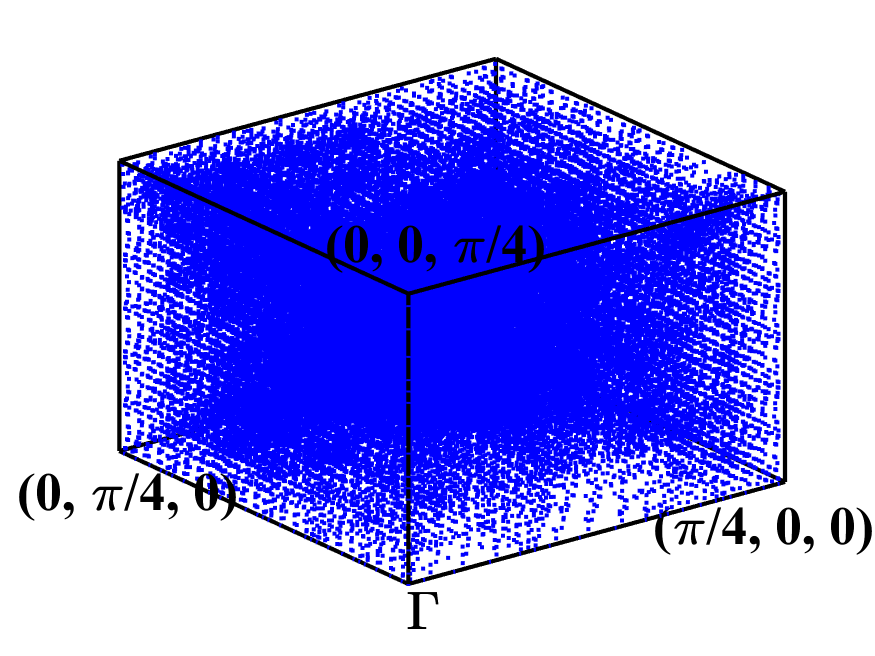}
}%
\quad
\centering
\subfigure[Loop 8]{\label{l6}
	\includegraphics[width=.3\textwidth,trim={0cm 0cm 0cm 0cm},clip]{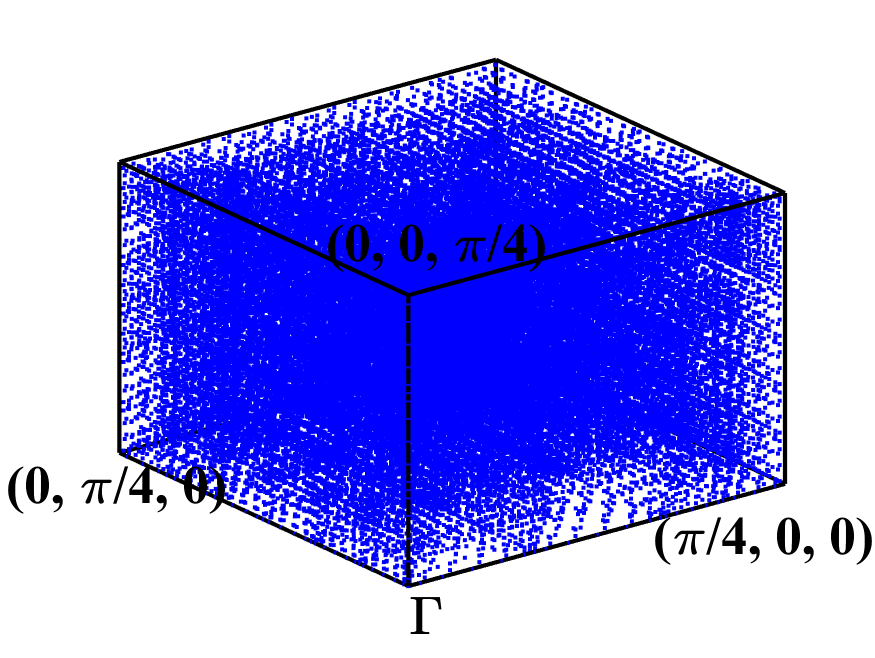}
}%
\\
\subfigure[Loop 3]{\label{distance l6}
	\includegraphics[width=.3\textwidth,trim={0cm 0cm 0cm 0cm},clip]{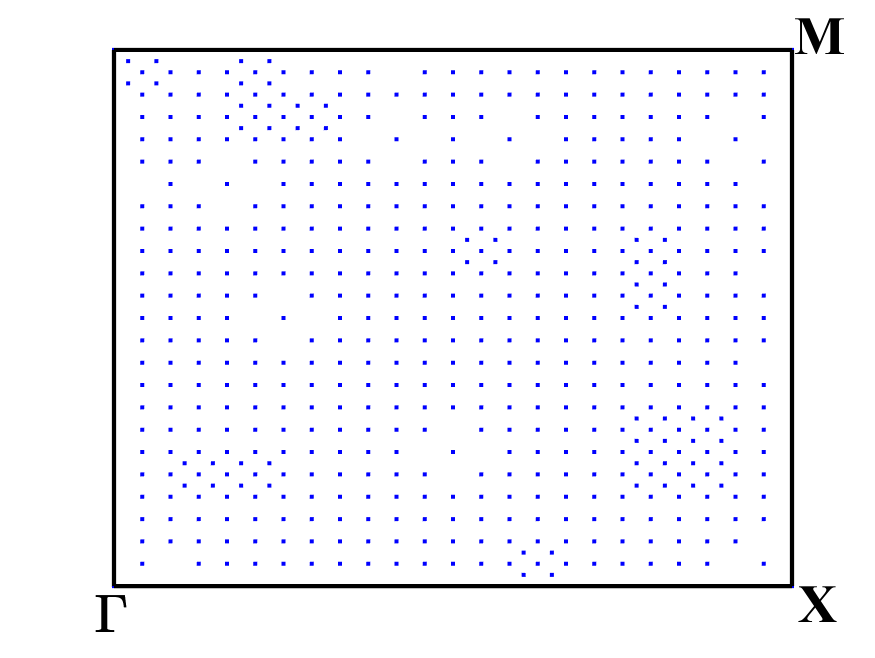}
}%
\quad
\centering
\subfigure[Loop 4]{\label{l7}
	\includegraphics[width=.3\textwidth,trim={0cm 0cm 0cm 0cm},clip]{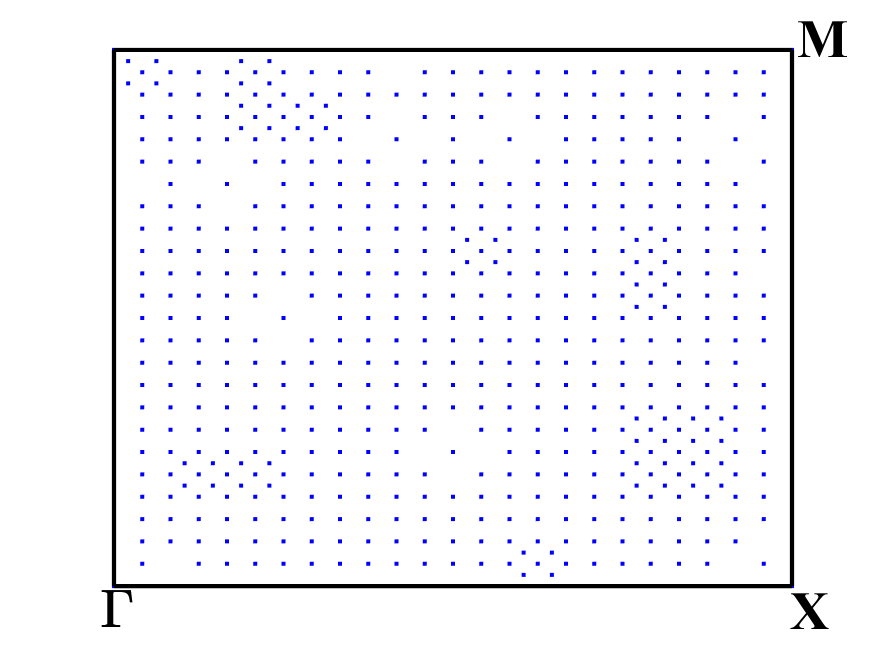}
}%
\quad
\subfigure[Loop 5]{\label{distance l7}
	\includegraphics[width=.3\textwidth,trim={0cm 0cm 0cm 0cm},clip]{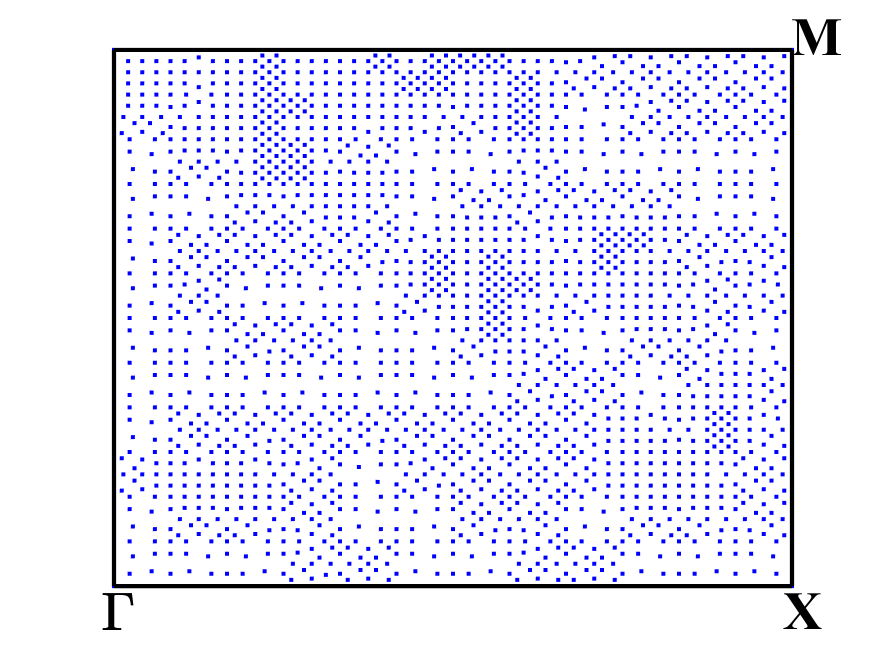}
}%
\quad
\centering
\subfigure[Loop 6]{\label{l8}
	\includegraphics[width=.3\textwidth,trim={0cm 0cm 0cm 0cm},clip]{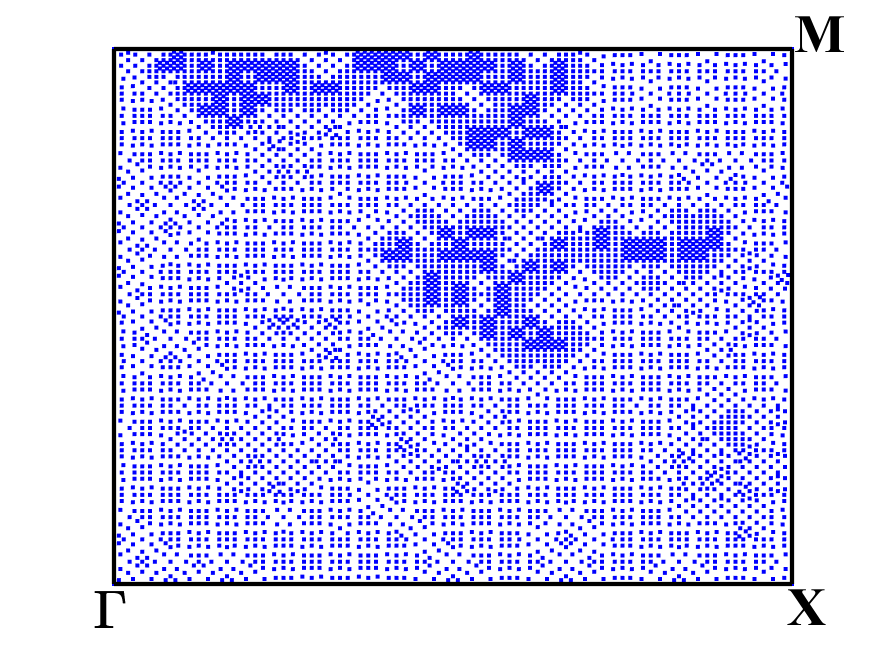}}%
\quad
\subfigure[Loop 7]{\label{distance l8}
	\includegraphics[width=.3\textwidth,trim={0cm 0cm 0cm 0cm},clip]{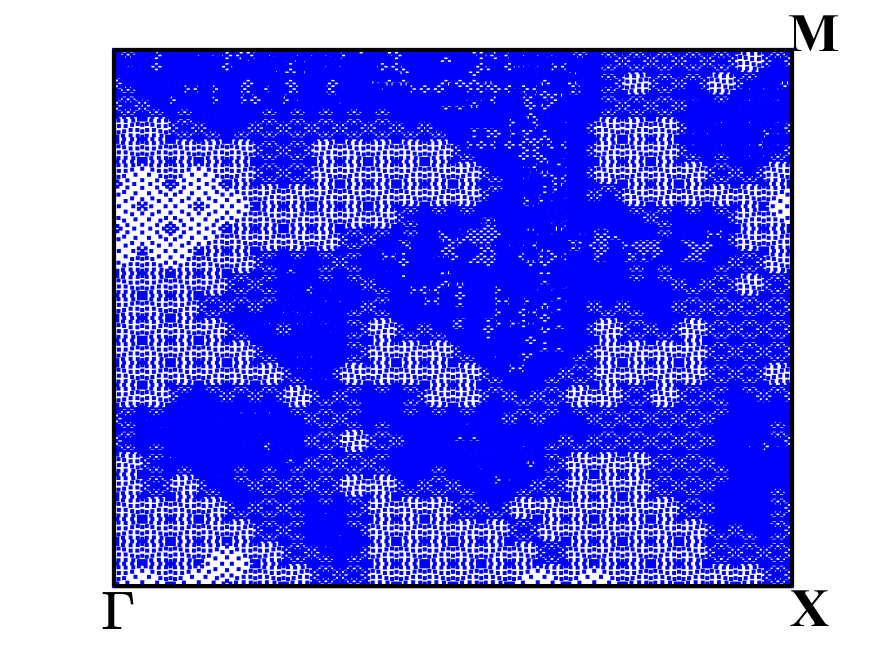}
	
}%
\quad
\subfigure[Loop 8]{\label{distance l9}
	\includegraphics[width=.3\textwidth,trim={0cm 0cm 0cm 0cm},clip]{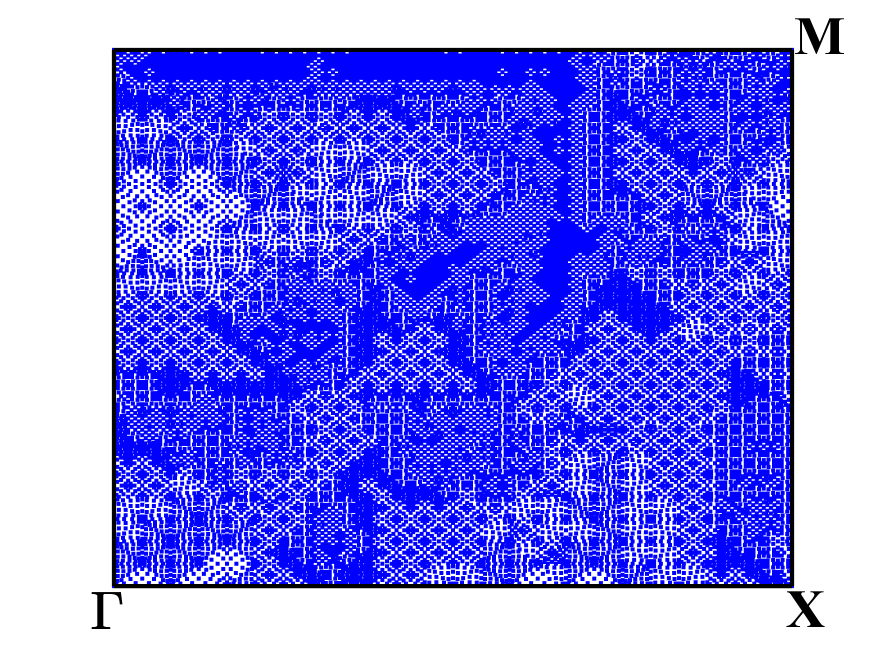}}
\\
\centering
\caption{Sampling points produced by Algorithm \ref{alg1} for Model 1 with $(\theta_1, \theta_2, \theta_3, \theta_4)=(0.1696a,0.1810a,0.2342a,0.3214a)$. (a)-(f): side view; (g)-(l): top view. We present a zoomed-in view in (c)-(f) for a better illustration.}
\label{points 1}
\end{figure}

\begin{figure}[hbt!]
	\centering
	\subfigure[Model 1 with optimal $\theta_1,\cdots,\theta_4$.]{\label{loglog maxi Model 1 optimize 4}
		\includegraphics[width=.35\textwidth,trim={0cm 0cm 0cm 0cm},clip]{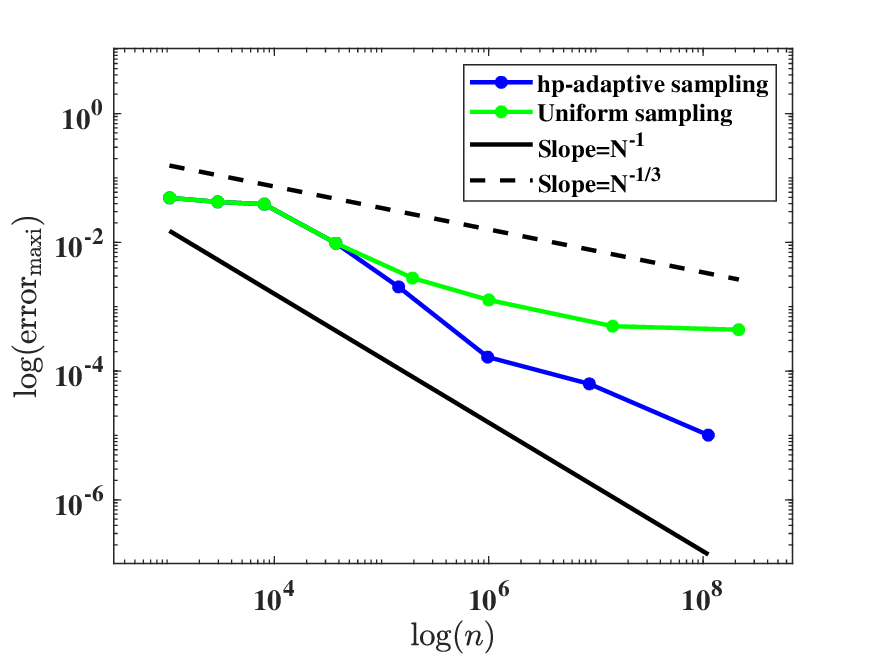}
	}%
	\quad
	\subfigure[Model 1 with optimal $\theta_1,\cdots,\theta_4$.]{\label{loglog average Model 1 optimize 4}
		\includegraphics[width=.35\textwidth,trim={0cm 0cm 0cm 0cm},clip]{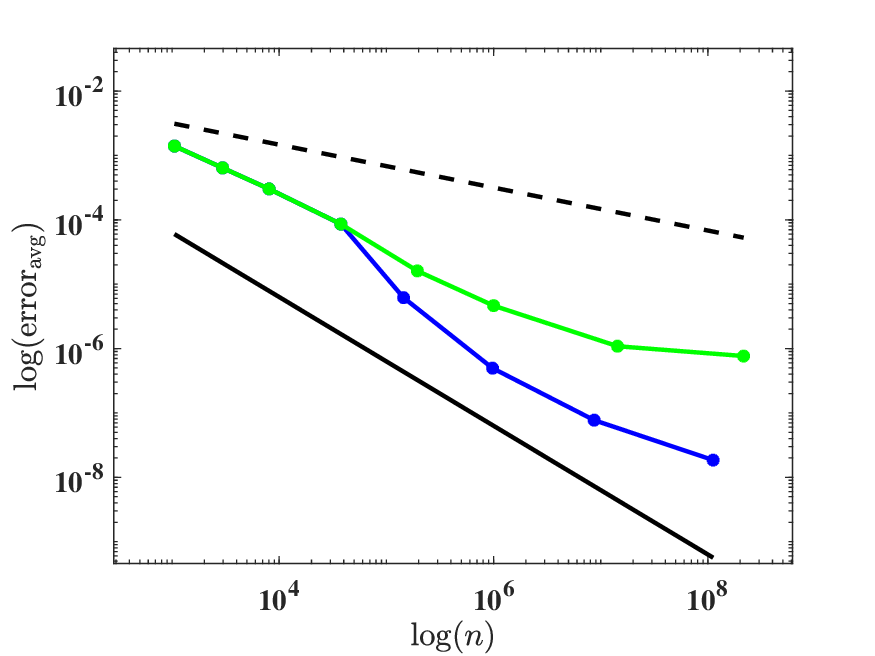}
	}%+
\\
	\centering
\subfigure[Model 1 with optimal $\theta_1,\theta_2$.]{\label{loglog maxi Model 1 optimize 2}
	\includegraphics[width=.35\textwidth,trim={0cm 0cm 0cm 0cm},clip]{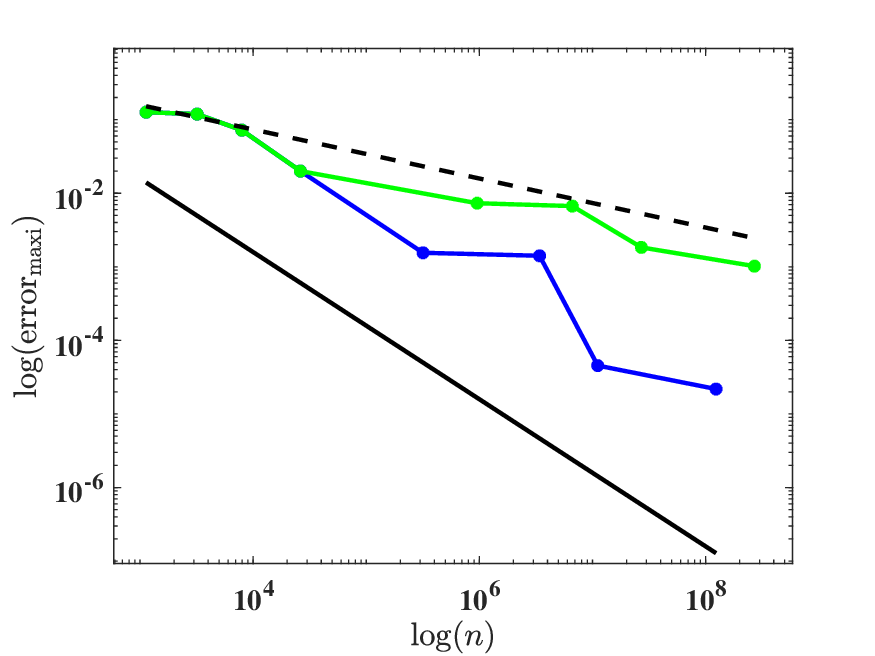}
}%
\quad
\subfigure[Model 1 with optimal $\theta_1,\theta_2$.]{\label{loglog average Model 1 optimize 2}
	\includegraphics[width=.35\textwidth,trim={0cm 0cm 0cm 0cm},clip]{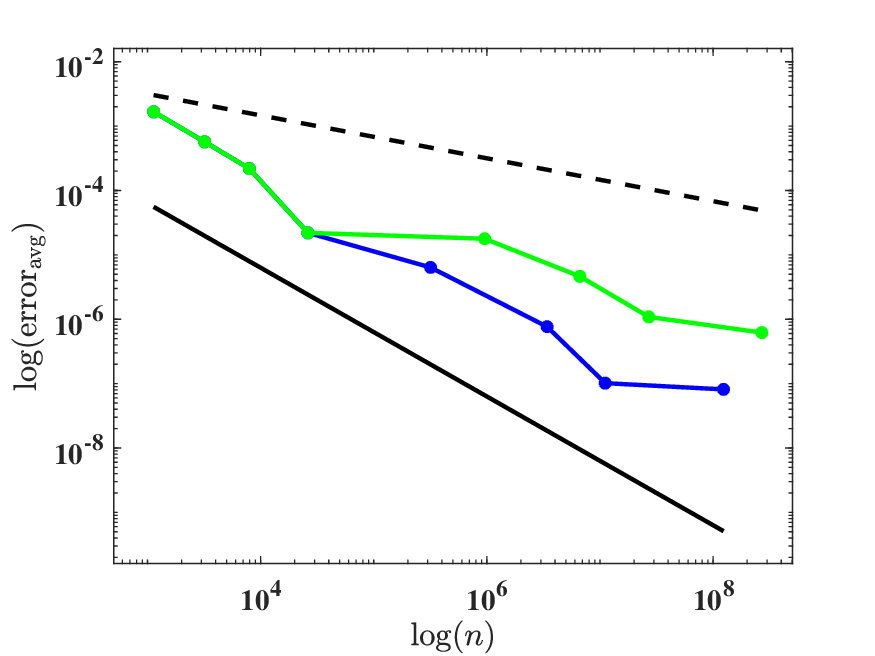}
}%
\\
	\centering
	\subfigure[Model 2 with optimal $\theta_1,\cdots,\theta_4$.]{\label{loglog maxi Model 2 optimize 4}
		\includegraphics[width=.35\textwidth,trim={0cm 0cm 0cm 0cm},clip]{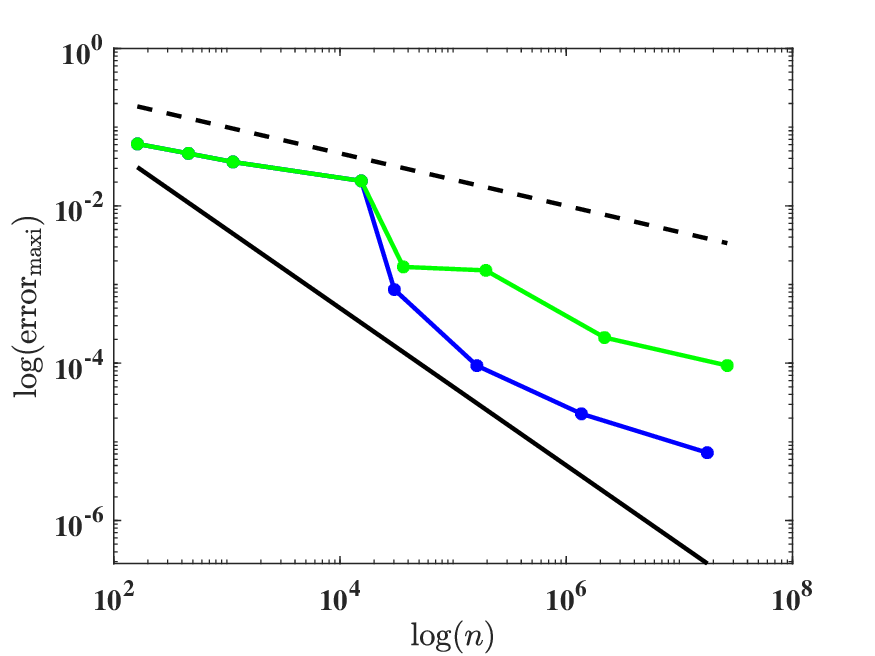}
	}%
	\quad
	\subfigure[Model 2 with optimal $\theta_1,\cdots,\theta_4$.]{\label{loglog average Model 2 optimize 4}
		\includegraphics[width=.35\textwidth,trim={0cm 0cm 0cm 0cm},clip]{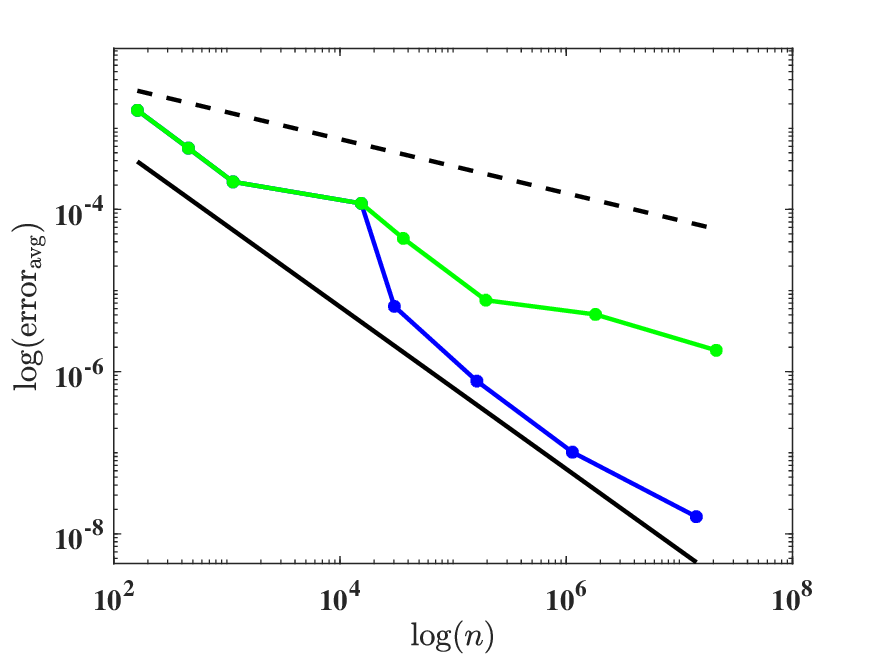}
	}%
	\centering
	\\
		\subfigure[Model 2 with optimal $\theta_1=\theta_2=\theta_3,\theta_4$.]{\label{loglog maxi Model 2 optimize 2}
		\includegraphics[width=.35\textwidth,trim={0cm 0cm 0cm 0cm},clip]{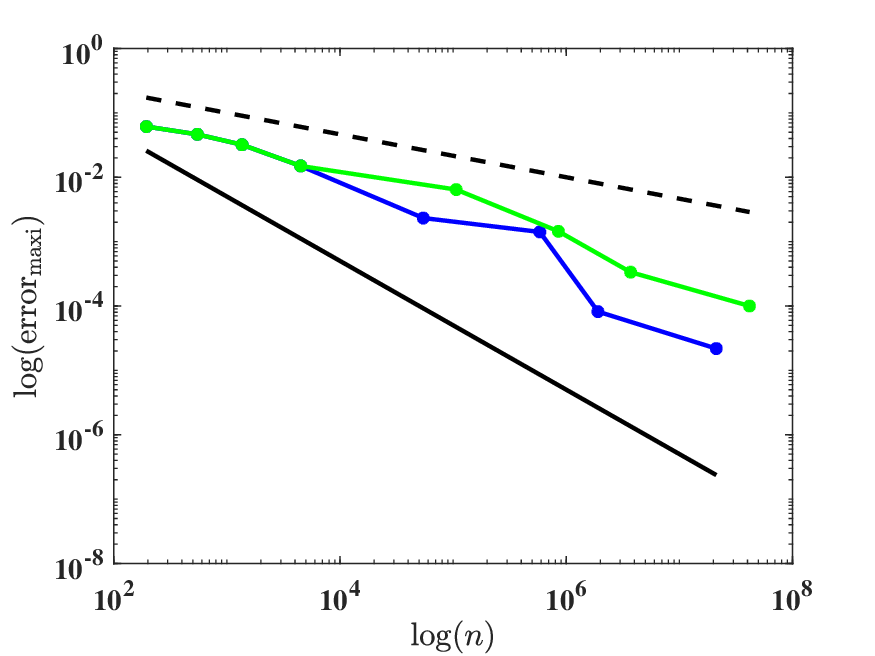}
	}%
	\quad
	\subfigure[Model 2 with optimal $\theta_1=\theta_2=\theta_3,\theta_4$.]{\label{loglog average Model 2 optimize 2}
		\includegraphics[width=.35\textwidth,trim={0cm 0cm 0cm 0cm},clip]{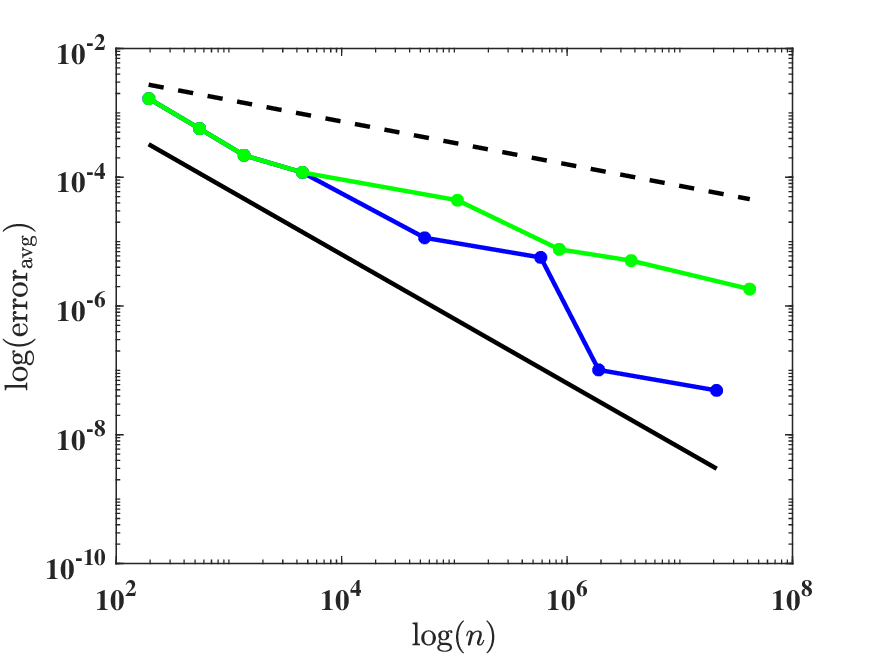}
	}%
	\centering
	\caption{Maximum (left) and average (right) relative error with respect to the number of sampling points.}
	\label{loglog Model 1&2}
\end{figure}

\subsection{Efficiency of Algorithm \ref{alg2}}\label{subsec:eff}
Finally, we show the efficiency of Algorithm \ref{alg2} by comparing with uniform mesh refinement with lowest-order element-wise interpolation under the optimal parameters. In both models, we randomly choose 5000 $\mathbf{k}$-points in the IBZ (denoted by $\hat{\mathcal{B}}$), on which we compute the reference solution using the $\mathbf{H}(\text{curl})$-conforming $\mathbf{k}$-modified N\'{e}d\'{e}lec edge FEM \eqref{discrete variational 2} with the mesh in Fig. \ref{mesh}. The pointwise relative error is then computed on $\hat{\mathcal{B}}$ by
\begin{align*}
	e_i(\mathbf{k}):=\frac{\vert \omega_i(\mathbf{k}) -\hat \omega_i(\mathbf{k})\vert}{{\omega_i}(\mathbf{k})} \quad \text{ for }\mathbf{k}\in\hat{\mathcal{B}},
\end{align*}
for some certain band number $i$. Here, $\omega_i(\mathbf{k})$ is the $i$th reference band function obtained directly by the $\mathbf{H}(\text{curl})$-conforming $\mathbf{k}$-modified N\'{e}d\'{e}lec edge FEM using the same mesh on the unit cell $\Omega$, and $\hat\omega_i(\mathbf{k})$ is obtained by the element-wise polynomial interpolation. We measure the accuracy of the methods using maximum and average relative errors
\begin{align*}
	\operatorname{error}_{\infty}:=\max_{i\in \mathbb{M}_q}\max_{\mathbf{k}\in\hat{\mathcal{B}}}\left|e_i(\mathbf{k})\right|\quad \mbox{and}\quad
	\operatorname{error}_{\text{avg}}:=\text{average}_{i\in \mathbb{M}_q}\max_{\mathbf{k}\in\hat{\mathcal{B}}}\left|e_i(\mathbf{k})\right|,
\end{align*}
where $q\in\{1,2\}$, $\mathbb{M}_1:=\{4,5\}$ for Model 1 and $\mathbb{M}_2:=\{2,3\}$ for Model 2.

We depict in Fig. \ref{points 1} the distribution of the sampling points for Model 1 with four optimizing parameters obtained from Algorithm \ref{alg1}. We observe that the sampling points are not uniformly distributed.
Moreover, Fig. \ref{loglog Model 1&2} presents the relative errors with respect to the number of sampling points, where $\operatorname{error}_{\infty}$ and $\operatorname{error}_{\text{avg}}$ versus $N$, i.e., the number of sampling points, on log-log coordinates with $\kappa=2\sqrt{2}$ and $\mu=1$ are plotted. We observe that Algorithm \ref{alg2} has a much faster convergence rate than the uniform one. Further, compared with the reference lines, Algorithm \ref{alg2} has an approximately first-order convergence rate, while the uniform method is much less efficient with an approximately $\frac{1}{3}$-order convergence rate, which is consistent with Theorem \ref{algebraic} and Remark \ref{uniform}.

\section{Conclusions}\label{sec:conclusion}
In this work, we present a generalization of an $hp$-adaptive sampling method for computing the band functions of three-dimensional photonic crystals with an application in the band gap maximization. We present theoretical analysis and numerical tests to verify its performance.
For the band gap maximization problem, we generate an adaptive mesh in the parameter domain each time when the geometry of the photonic crystal is updated. One intriguing question is whether one can use the same adaptive mesh for a few updates in the geometry of the photonic crystals in order to make our algorithm more efficient. Also, it is of interest to have more freedom in constructing photonic crystals with an infinite number of design parameters, e.g., the interface between two different materials can be designed.

\bibliographystyle{abbrv}
\bibliography{refer2}
\end{document}